\documentclass[reqno,12pt]{amsart}
\pdfoutput=1
\usepackage{amsmath,amsthm,amsfonts,amssymb,ifpdf}
\usepackage{amssymb}
\usepackage{amsmath}
\usepackage{amsthm}
\usepackage{graphicx}
\usepackage[all]{xy}
\usepackage{enumerate}
\usepackage{tikz-cd}

\newcommand{\RomanNumeralCaps}[1]
    {\MakeUppercase{\romannumeral #1}}
\theoremstyle{plain} 
\newtheorem{theorem}{Theorem}[section]
\newtheorem{proposition}[theorem]{Proposition}
\newtheorem{lemma}[theorem]{Lemma}
\newtheorem{corollary}[theorem]{Corollary}

\theoremstyle{definition} \newtheorem{definition}[theorem]{Definition}

\theoremstyle{remark} \newtheorem{remark}[theorem]{Remark}

\usepackage[all]{xy}

\ifpdf
  \usepackage[
    pdftex,
    colorlinks,%
    linkcolor=blue,citecolor=red,urlcolor=blue,
    hyperindex,%
    plainpages=false,%
    bookmarksopen,%
    bookmarksnumbered%
  ]{hyperref} 
 
 \usepackage{thumbpdf}
\else
  \usepackage{hyperref}
\fi

\def\deg{{\rm deg}}
\def\Sym{{\rm Sym}}
\def\Hilb{{\rm Hilb}}
\def\Quot{{\rm Quot}}
\def\Ker{{\rm Ker}}
\def\Pic{{\rm Pic}}
\def\Div{{\rm Div}}

\def\lim{{\rm lim}}

\newcommand{\ncom}{\newcommand}
\ncom{\mylabel}[1]{{\rm (#1)}\label{#1}}
\ncom{\Hom}{{\textit{Hom}}}
\ncom{\eop}{{\hfill $\Box$}}
\begin{document}
\baselineskip=16pt


\setcounter{tocdepth}{1}

\title[Diagonal property and weak point property]{Diagonal property and weak point property of higher rank divisors and certain Hilbert schemes}

\author{Arijit Mukherjee}
\address{Department of Mathematics, Indian Institute of Science Education and Research Tirupati, Srinivasapuram, Venkatagiri Road, Jangalapalli Village, Panguru (G.P), Yerpedu Mandal, Tirupati District, Andhra Pradesh - 517 619, India.}
\email{mukherjee7.arijit@gmail.com}

\author{D S Nagaraj}
\address{Department of Mathematics, Indian Institute of Science Education and Research Tirupati, Srinivasapuram, Venkatagiri Road, Jangalapalli Village, Panguru (G.P), Yerpedu Mandal, Tirupati District, Andhra Pradesh - 517 619, India.}
\email{dsn@labs.iisertirupati.ac.in}

\begin{abstract}
In this paper, we introduce the notion of the diagonal property and the weak point property for an ind-variety.  We prove that the ind-varieties of higher rank divisors of integral slopes on a smooth projective curve have the weak point property.  Moreover, we show that the ind-variety of $(1,n)$-divisors has the diagonal property and is a locally complete linear ind-variety and calculate its Picard group.  Furthermore, we obtain that the Hilbert schemes of a curve associated to the good partitions of a constant polynomial satisfy the diagonal property.  In the process of obtaining this, we provide the exact number of such Hilbert schemes up to isomorphism by proving that the multi symmetric products associated to two distinct partitions of a positive integer $n$ are not isomorphic.
\end{abstract}
\maketitle

\tableofcontents

\section{Introduction}
\label{sec:Introduction}
Let $X$ be a smooth projective variety over the field of complex numbers.  By the diagonal subscheme of $X$, denoted by $\Delta_X$, one means the image of the embedding $\delta: X\rightarrow X\times X$ given by $\delta(x)=(x,x)$, where $x\in X$.  This subscheme plays a central role in intersection theory.  In fact, to get hold of the fundamental classes of any subschemes of a variety $X$, it's enough to know the fundamental class of the diagonal $\Delta_X$ of $X$, (cf. \cite{P}).  
 
In this paper, we talk about the diagonal property and the weak point property of some varieties.  Broadly speaking, the diagonal property of a variety $X$ is a property which demands a special structure of the diagonal $\Delta_X$ and therefore very significant to study from the viewpoint of intersection theory.  Moreover, being directly related to the diagonal subscheme $\Delta_X$, this property imposes strong conditions on the variety $X$ itself.  For example, this property is responsible for the existence or non-existence of cohomologically trivial line bundles on $X$.  The weak point property is also very much similar to diagonal property but a much weaker one.  Both of these notions were introduced in \cite{PSP}.  Many mathematicians have studied about the diagonal property and the weak point property of varieties, (cf. \cite{F}, \cite{FP}, \cite{LS}).  In this paper, we introduce these two notions for an ind-variety, that is an inductive system of varieties and showed that the ind-varieties of higher rank divisors of integral slope on a smooth projective curve $C$  satisfy these properties.  Also, we show that some Hilbert schemes associated to good partitions of a constant polynomial satisfy the diagonal property.

Before mentioning the results obtained in this paper more specifically, let us fix some notations which we are going to use repeatedly.  We denote by $\mathbb{C}$ the field of complex numbers.  In this paper, by $C$ we always mean a smooth projective curve over $\mathbb{C}$. The notation $\mathcal{O}_C$ is reserved for the structure sheaf over $C$.  For a given divisor $D$ on $C$, by $\mathcal{O}_C(D)$ we mean the corresponding line bundle over $C$ and denote its degree by $\deg(D)$.  By $\Sym^d(C)$ and $J(C)$ we denote the $d$-th symmetric power of the curve $C$ and the Jacobian variety of degree $0$ line bundles on $C$ respectively.  For a given positive integer $n$ and a locally free sheaf (equivalently, a vector bundle) $\mathcal{F}$ over $C$, by $\mathcal{F}^n$ we mean the direct sum of $n$ many copies of $\mathcal{F}$.  By $\Quot^{d}_{\mathcal{G}}$ we denote the Quot scheme parametrizing all torsion quotients of $\mathcal{G}$ having degree $d$, $\mathcal{G}$ being any coherent sheaf on $C$.  For a given polynomial $P(t)\in \mathbb{Q}[t]$, we denote the Quot scheme parametrizing all torsion quotients of $\mathcal{G}$ having Hilbert polynomial $P(t)$ by $\Quot^{P}_{\mathcal{G}}$.
 
Let us now go through the chronology of this paper in a bit more detail.  The manuscript is arranged as follows.  In Section \ref{sec: dp and wpp}, we recall the definitions of the  diagonal property and the weak point property for a smooth projective variety and talk about a relation between these two properties.  Moreover, for a smooth projective curve $C$ over $\mathbb{C}$, we recall a couple of relevant results about the variety $\Sym^d(C)$ and the Quot scheme $\Quot^{d}_{\mathcal{O}_C^n}$.  In Section \ref{sec: Higher rank divisors}, we recall the definition of $(r,n)$-divisors on $C$ \& the ind-variety made out of such divisors.  We then precisely define, what we mean by the diagonal property and the weak point property of an ind-variety and prove the following theorems followed by that.
\begin{theorem}\label{Main_Theorem_1}
Let $C$ be a smooth projective curve over $\mathbb{C}$.  Also let $r\geq 1$ and $n$ be two integers.  Then the ind-variety of $(r,n)$-divisors having integral slope on $C$ has the weak point property. 
\end{theorem} 
\begin{theorem}\label{Main_Theorem_2}
Let $C$ be a smooth projective curve over $\mathbb{C}$ and $n$ any given integer. Then the ind-variety of $(1,n)$-divisors on $C$ has the diagonal property.
\end{theorem} 
We end Section \ref{sec: Higher rank divisors} by showing some more properties of the ind-variety of $(1,n)$-divisors on $C$, as in \ref{Main_Theorem_2}, which are very much useful in the context of studying  Barth-Van de Ven-Tyurin-Sato theorem (cf. \cite{PT}).  We obtain the following theorem to be precise. 
\begin{theorem}\label{Main Theorem_3}
Let $C$ be a smooth projective curve over $\mathbb{C}$ and $n$ any given integer. Then the ind-variety of $(1,n)$-divisors on $C$ is a locally complete linear ind-variety.
\end{theorem} 
As an immediate consequence, we calculate the Picard variety of this ind-variety.
\begin{corollary}\label{Corollary regarding Picard group of the ind-variety}
Let $C$ be a smooth projective curve over $\mathbb{C}$ and $n$ any given integer. Then the Picard group of the ind-variety of $(1,n)$-divisors on $C$ is $\Pic(J(C))\oplus \mathbb{Z}$.
\end{corollary}         
In Section \ref{sec: Hilbert scheme}, we deal with the Hilbert scheme of a curve associated to a polynomial $P$ and its good partition.  E.~Bifet has dealt with these schemes in \cite{B}.  Moreover, he showed that the Quot scheme $\Quot_{\mathcal{O}_C^r}^P$ can be written as disjoint union of some smooth, the torus $\mathbb{G}_m^r$-invariant, locally closed vector bundles over the mentioned Hilbert schemes.  Here, we talk about the diagonal property of such Hilbert schemes and found the exact number of such schemes.  Towards that, we first prove the following lemma.
\begin{lemma}\label{lemma_1}
Let $n$ be a given positive integer.  Then any partition of $n$ is also a good partition of $n$ and vice versa.
\end{lemma} 
In the Lemma \ref{lemma_1}, we interpret the integer $n$ as a constant polynomial and therefore it makes sense to talk about good partition of $n$.  We then deal with the products of projective spaces corresponding to distinct partitions of same length of a given integer $n$ and show the following:    

\begin{proposition}\label{prop_1}  
Let $n$ be a positive integer.  Let $(m_1,m_2,\ldots,m_s)$ and $(n_1,n_2,\ldots,n_s)$ be two distinct partitions of $n$ of same length $s$.  Then $\mathbb{P}^{m_1}\times 
\mathbb{P}^{m_2}\times \cdots \times \mathbb{P}^{m_s}$ is not isomorphic to  $\mathbb{P}^{n_1}\times \mathbb{P}^{n_2}\times \cdots \times \mathbb{P}^{n_s}$.
\end{proposition}

By a multi symmetric product of $C$ of type $[(n_1,n_2,\ldots,n_r),n]$ we mean the product $\Sym^{n_1}(C)\times \Sym^{n_2}(C)\times \cdots \times \Sym^{n_r}(C)$, $(n_1,n_2,\ldots,n_r)$ being a partition of $n$.  Then we look upon the multi symmetric products corresponding to partitions of different lengths and prove that they are not isomorphic by showing that their first Betti number differ.  Specifically, we obtain :
\begin{proposition}\label{prop_2}
Let $C$ be a smooth projective curve over $\mathbb{C}$ of genus $g$ with $g\geq 1$.  Let $n$ be a positive integer, and $(n_1,n_2,\ldots,n_r)$ and $(m_1,m_2,\ldots,m_s)$ two distinct partitions of $n$ of different lengths.  Then the multi symmetric product of $C$ of type $[(n_1,n_2,\ldots,n_r),n]$ and $[(m_1,m_2,\ldots,m_s),n]$ are not isomorphic.
\end{proposition}

Using Lemma \ref{lemma_1}, Proposition \ref{prop_1} and Proposition \ref{prop_2}, we obtain the following theorem :
\begin{theorem}\label{Main_Theorem_4}
Let $C$ be a smooth projective curve over $\mathbb{C}$ and $n$ a positive integer.  Let $p(n)$ denote the number of partitions of $n$.  Then the following hold:
\begin{enumerate}
\item There are at most $p(n)$ many Hilbert schemes $\Hilb^{\underline{n}}_C$ (up to isomorphism) associated to the constant polynomial $n$ and its good partitions $\underline{n}$ satisfying diagonal property.
\item Moreover, this upper bound is attained by any genus $0$ curve $C$ and hence is sharp.
\item Furthermore, for $n=1,2,3$, the upper bound is attained by any curve $C$.
\end{enumerate}
\end{theorem} 

We further look at the multi symmetric products corresponding to distinct partitions of same length of a given integer and check whether they are isomorphic or not.  In that context, we obtain: 
\begin{proposition}\label{prop_3}
Let $C$ be a smooth projective curve over $\mathbb{C}$ of genus $g$ with $g\geq 1$.  Let $n$ be a positive integer, and $(n_1,n_2,\ldots,n_r)$ and $(m_1,m_2,\ldots,m_r)$ two distinct partitions of $n$ of same length.  Then the multi symmetric product of $C$ of type $[(n_1,n_2,\ldots,n_r),n]$ and $[(m_1,m_2,\ldots,m_r),n]$ are not isomorphic.
\end{proposition}
We prove Proposition \ref{prop_3} by breaking it down into two cases, namely $\min\{n_r,m_r\}\leq 2g-1$ and $\min\{n_r,m_r\}\geq 2g-1$.  For the first case, we prove using Betti numbers of the involved multi symmetric products (cf. Proposition \ref{non isomorphic multi symmetric products corresponding to partitions of same length having smaller parts}).  We use the projective bundle nature of symmetric products for the later case (cf. Proposition \ref{non isomorphic multi symmetric products corresponding to partitions of same length having bigger parts}).

Using Proposition \ref{prop_3}, we further strengthen Theorem \ref{Main_Theorem_4} to the maximum possible extent.  Precisely, we obtain :
\begin{theorem}\label{Main_Theorem_5}
Let $C$ be a smooth projective curve over $\mathbb{C}$ and $n$ a positive integer.  Let $p(n)$ denote the number of partitions of $n$.  Then there are exactly $p(n)$ many Hilbert schemes $\Hilb^{\underline{n}}_C$ (up to isomorphism) associated to the constant polynomial $n$ and its good partitions $\underline{n}$ satisfying diagonal property.
\end{theorem}

\section{On the diagonal property and the weak point property of a variety}\label{sec: dp and wpp} 
In this section, we recall the notions of the diagonal property and the weak point property of a variety and talk about relations between these two properties.  Moreover, for a smooth projective curve $C$ over $\mathbb{C}$, we recall a couple of relevant results about the variety $\Sym^d(C)$ and the Quot scheme $\Quot^{d}_{\mathcal{O}_C^n}$.   

Let us begin with the precise definitions of the diagonal property and the weak point property of a variety.
\begin{definition}\label{dp}
Let $X$ be a variety over the field of complex numbers.  Then $X$ is said to have the diagonal property if there exists a vector bundle $E\rightarrow X\times X$ of rank equal to the dimension of $X$, and a global section $s$ of $E$ such that the zero scheme $Z(s)$ of $s$ coincides with the diagonal $\Delta_X$ in $X\times X$.  
\end{definition} 
\begin{definition}\label{wpp}
Let $X$ be a variety over the field of complex numbers.  Then $X$ is said to have the weak point property if there exists a vector bundle $F\rightarrow X$ of rank equal to the dimension of $X$, and a global section $t$ of $F$ such that the zero scheme $Z(s)$ of $s$ is a reduced point of $X$.
\end{definition}
\begin{remark}\label{dp implies wpp}
It can be noted immediately that for a variety, having the weak point property is in fact a weaker condition than having the diagonal property.  To prove this precisely, let's stick to the notations of Definition \ref{dp} and \ref{wpp}.  Let us choose a point $x_0\in X$.  Then $Z(s|_{X\times \{x_0\}})=\{x_0\}$.  Therefore, the diagonal property implies the weak point property. 
\end{remark}
We now quickly go through some results related to the diagonal property and the weak point property of two varieties which arise very naturally from a given curve $C$.  To be specific, we look upon the varieties $\Sym^d(C)$ and $\Quot^{d}_{\mathcal{O}_C^n}$.  We mention a couple of results in this context.  These are due to \cite{BS}.
\begin{theorem}\label{symm prod_dp}
Let $C$ be a smooth projective curve over $\mathbb{C}$.  Then, the $d$-th symmetric product $\Sym^d(C)$ of the curve $C$ has the diagonal property for any positive integer $d$.
\end{theorem}
\begin{proof}
See \cite[Theorem 3.1, p. 447]{BS}.
\end{proof}
\begin{theorem}\label{Quot scheme_wpp}
Let $C$ be a smooth projective curve over $\mathbb{C}$. Let $d$ and $n$ be two given positive integer such that $n|d$.  Then the Quot scheme $\Quot^{d}_{\mathcal{O}_C^n}$ parametrizing the torsion quotients of $\mathcal{O}_C^n$ of degree $d$ has the weak point property.    
\end{theorem}
\begin{proof}
See \cite[Theorem 2.2, p. 446-447]{BS}.
\end{proof}  
\begin{remark}\label{positivity required in the hypothesis}
Let us discuss about the hypothesis of Theorem \ref{Quot scheme_wpp}.  Firstly, positivity of the integer $n$ is necessary as we are talking about the sheaf $\mathcal{O}_C^n$.  Moreover, if we assume that $d$ is a positive integer and $n|d$, then there exists a positive integer $r$ such that $d=nr$.  The positivity of this integer $r$ is heavily used in the proof.  Indeed, the authors first showed that to prove Theorem \ref{Quot scheme_wpp}, it is enough to show that the Quot scheme $\Quot^{d}_{L^n}$ has the weak point property for some degree $r$ line bundle $L$ over $C$.  Now the line bundle $L$ is taken to be the line bundle $\mathcal{O}_C(rx_0)$, where $x_0\in X$.  Now positivity of $r$ gives the natural inclusion $i:\mathcal{O}_C \hookrightarrow \mathcal{O}_C(rx_0)$.  This in turn gives the following short exact sequence:
\begin{equation}\label{ses_wpp of Quot scheme}
0\rightarrow \mathcal{O}_C^n \rightarrow \mathcal{O}_C(rx_0)^n \rightarrow T \rightarrow 0.
\end{equation}   
Now the torsion sheaf $T$ as in \eqref{ses_wpp of Quot scheme} lies in the sheaf $\Quot^{d}_{{\mathcal{O}_C(rx_0)}^n}$, the sheaf they wanted to work on to prove the required result.  So, positivity of $d$ has a huge role to play in the proof. 
\end{remark}
\begin{remark}
It is worthwhile to note a connection between Theorem \ref{symm prod_dp} \& \ref{Quot scheme_wpp}.  If we take, $n=1$, then Theorem \ref{Quot scheme_wpp} says that for any positive integer $d$, the Quot scheme $\Quot^{d}_{\mathcal{O}_C}$ has the weak point property.  As, $\Sym^d(C)\simeq \Quot^{d}_{\mathcal{O}_C}$, by Remark \ref{dp implies wpp}, Theorem \ref{Quot scheme_wpp} follows from Theorem \ref{symm prod_dp} for $n=1$ case.
\end{remark}
\section{Higher rank divisors on a curve, corresponding ind-varieties and the diagonal \& the weak point property}
\label{sec: Higher rank divisors}
In this section, we recall the definition of higher rank divisors on a curve, corresponding ind-varieties and quasi-isomorphism between them.  Then we introduce the notion of the diagonal property and the weak point property for an ind-variety in general and prove some results about the ind-varieties of higher rank divisor in particular.   
 
Let us denote by $K$ the field of rational functions on $C$, thought as a constant $\mathcal{O}_C$-module.
\begin{definition}
	A divisor of rank $r$ and degree $n$ over $C$ is a coherent sub $\mathcal{O}_C$-module of $K^{\oplus r}$ having rank $r$ and degree $n$.  This is denoted by $(r,n)$-divisor.
\end{definition}
\begin{remark}
Since we take $C$ to be smooth, these $(r,n)$-divisors coincide with the matrix divisors defined by A. Weil, (cf. \cite{W}).
\end{remark}
Let us denote the set of all $(r,n)$-divisors on $C$ by $\Div^{r,n}$.  Let $D$ be an effective divisor of degree $d$ over $C$.  Then corresponding to $D$, let us define the following subset of $\Div^{r,n}$, denoted by $\Div^{r,n}(D)$ as follows:
\begin{definition}
	$\Div^{r,n}(D):=\{E \in \Div^{r,n} \mid E \subseteq \mathcal{O}_{C}(D)^{\oplus r}\}$.
\end{definition}
Then clearly we have, $\Div^{r,n}=\bigcup_{D\geq 0}\Div^{r,n}(D)$.
Also, the elements of $\Div^{r,n}(D)$ can be identified with the rational points of the Quot scheme $\Quot_{\mathcal{O}_C(D)^r}^m$, where $m=r\cdot \deg (D)-n$.  Therefore taking $D=\mathcal{O}_C$, we can say that the elements of $\Div^{r,n}(\mathcal{O}_C)$ can be identified with the rational points of the Quot scheme $\Quot_{\mathcal{O}_C^r}^{-n}$.

Let us now recall what one means by a inductive system of varieties.
\begin{definition}
	An ind-variety $\mathbf{X}=\{X_\lambda,f_{\lambda \mu}\}_{\lambda,\mu \in \Lambda}$ is an inductive system of complex algebraic varieties $X_\lambda$ indexed by some filtered ordered set $\Lambda$.  That is to say, an ind-variety is a collection $\{X_\lambda\}_{\lambda \in \Lambda}$ of complex algebraic varieties, where $\Lambda$ is some filtered ordered set, along with the morphisms $f_{\lambda \mu}: X_{\lambda}\rightarrow X_{\mu}$ of varieties for every $\lambda \leq \mu$ such that the following diagrams commute for every $\lambda \leq \mu \leq \nu$.
\begin{equation*}
\xymatrix{
X_{\lambda} \ar[rd]_{f_{\lambda \nu}}\ar[r]^{f_{\lambda \mu}} & X_{\mu}
		\ar[d]^{f_{\mu \nu}}\\ 
		 & X_{\nu}}
\end{equation*}	
\end{definition}
Taking the indexing set $\Lambda$ to be the set of effective divisors on $C$, we have the inclusion 
\begin{equation}\label{E6}
\Div^{r,n}(D_\alpha)\rightarrow \Div^{r,n}(D_\beta),
\end{equation}
induced by the closed immersion $\mathcal{O}_{C}(D_{\alpha})^{\oplus r}\hookrightarrow \mathcal{O}_{C}(D_{\beta})^{\oplus r}$ for any pair of effective divisors $D_\alpha,D_\beta$ satisfying $D_\alpha\leq D_\beta$.
\begin{definition}\label{Div_ind-variety}
	The ind-variety determined by the inductive system consisting of the varieties $\Div^{r,n}(D)$ and the closed immersions as in (\ref{E6}) is denoted by ${\mathbf{Div}}^{r,n}$.
\end{definition}

Now we are going to consider another ind-variety.  Given any effective divisor $D$ on $C$, we consider a complex algebraic variety $Q^{r,n}(D)$ defined as follows.
\begin{definition}\label{Quot schemes as constituent of ind-variety}
	$Q^{r,n}(D):=\Quot^{n+r\cdot \deg (D)}_{\mathcal{O}_C^r}$.
\end{definition}
Let $D_1$ and $D_2$ be any two effective divisors  with $D_2\geq D_1$.   Denoting $D_2-D_1$ as $D$, we have the following structure map denoted by $\mathcal{O}_C(-D)$.
\begin{equation*}
\mathcal{O}_C(-D): \Quot^{n+r\cdot \deg (D_1)}_{\mathcal{O}_C^r}\rightarrow \Quot^{n+r\cdot \deg (D_2)}_{\mathcal{O}_C^r},
\end{equation*}
where the map $\mathcal{O}_C(-D)$ means tensoring the submodules with $\mathcal{O}_C(-D)$.
Elaborately, let $(\mathcal{F},q) \in \Quot^{n+r\cdot \deg (D_1)}_{\mathcal{O}_C^r}$.  Therefore we have the following exact sequence:
\begin{equation*}
\xymatrix{ 0 \ar[r]& \Ker (q) \ar[r] &\mathcal{O}_C^r\ar[r]^{q} &\mathcal{F}
	\ar[r] &0},
\end{equation*}
where degree of $\mathcal{F}$ is $n+r\cdot \deg (D_1)$ and hence degree of $\Ker (q)$ is $-n-r\cdot \deg (D_1)$.  Tensoring this by $\mathcal{O}_C(-D)$ we get,
\begin{equation*}
\xymatrix{ 0 \ar[r]& \Ker (q)\otimes \mathcal{O}_C(-D) \ar[r] &\mathcal{O}_C(-D)^r\ar[r] &\mathcal{F}\otimes \mathcal{O}_C(-D)
	\ar[r] &0}.
\end{equation*}
Here $\deg(\Ker (q)\otimes \mathcal{O}_C(-D))=r\cdot(\deg (D_1)-\deg (D_2))-n-r\cdot \deg (D_1)=-n-r\cdot \deg (D_2)$.  Now as $\mathcal{O}_C(-D)^r$ sits inside $\mathcal{O}_C^r$, $\Ker (q)\otimes \mathcal{O}_C(-D)$ also sits inside $\mathcal{O}_C^r$.  Therefore we now get the following exact sequence:
\begin{equation*}
\xymatrix{ 0 \ar[r]& \Ker (q)\otimes \mathcal{O}_C(-D) \ar[r] &\mathcal{O}_C^r\ar[r]^{q_1} &\mathcal{F}_1
	\ar[r] &0},
\end{equation*} 
where $\deg(\mathcal{F}_1)=n+r\cdot \deg (D_2)$. Hence, $\mathcal{F}_1 \in \Quot^{n+r\cdot \deg (D_2)}_{\mathcal{O}_C^r}$.  Thus, the map $\mathcal{O}_C(-D): \Quot^{n+r\cdot \deg (D_1)}_{\mathcal{O}_C^r} \rightarrow \Quot^{n+r\cdot \deg (D_2)}_{\mathcal{O}_C^r}$ given by $(\mathcal{F},q) \mapsto (\mathcal{F}_1,q_1)$ is well defined.  Therefore for $D_2\geq D_1$ we have,
\begin{equation}\label{E7}
\mathcal{O}_C(-D):Q^{r,n}(D_1)\rightarrow Q^{r,n}(D_2).
\end{equation}
\begin{definition}\label{Quot_ind-variety}
	The ind-variety determined by the inductive system consisting of the varieties $Q^{r,n}(D)$ and the morphisms as in (\ref{E7}) is denoted by ${\mathbf{Q}}^{r,n}$.
\end{definition}

Let us clarify what we mean by a good enough morphism in the category of ind-varieties.   
\begin{definition}
	Let $\mathbf{X}=\{X_D,f_{DD_1}\}_{D,D_1\in \mathcal{D}}$ and $\mathbf{Y}=\{Y_D,g_{DD_1}\}_{D,D_1\in \mathcal{D}}$ be two inductive system of complex algebraic varieties, where $\mathcal{D}$ is the ordered set of all effective divisors on $C$.  Then by a morphism $\mathbf{\Phi}=\{\alpha,\{\phi_D\}_{D\in \mathcal{D}}\}$ from $\mathbf{X}$ to $\mathbf{Y}$ we mean an order preserving map $\alpha:\mathcal{D}\rightarrow \mathcal{D}$ together with a family of morphisms $ \phi_D: X_D \rightarrow Y_{\alpha(D)}$ satisfying the following commutative diagrams for all $D,D_1\in \mathcal{D}$ with $D\leq D_1$.
	\begin{equation*}
	\label{eq:277}
	\xymatrix{X_D \ar[d]_{f_{DD_1}}\ar[rr]^{\phi_D} && Y_{\alpha(D)}
		\ar[d]^{g_{\alpha(D)\alpha(D_1)}}\\ 
		X_{D_1}\ar[rr]^{\phi_{D_1}}
		&& Y_{\alpha(D_1)}
	}
	\end{equation*}	
\end{definition}
\begin{remark}
Note that $\alpha:\mathcal{D}\rightarrow \mathcal{D}$ being an order preserving map, $D\leq D_1\Rightarrow \alpha(D)\leq \alpha(D_1)$.  Therefore the map $g_{\alpha(D)\alpha(D_1)}:Y_{\alpha(D)}\rightarrow Y_{\alpha(D_1)}$ makes sense.
\end{remark}
\begin{definition}
	Let $\mathbf{X}=\{X_D,f_{DD_1}\}_{D,D_1\in \mathcal{D}}$ and $\mathbf{Y}=\{Y_D,g_{DD_1}\}_{D,D_1\in \mathcal{D}}$ be two inductive system of complex algebraic varieties.  Then a morphism $\mathbf{\Phi}=\{\alpha,\{\phi_D\}_{D\in \mathcal{D}}\}$ from $\mathbf{X}$ to $\mathbf{Y}$ is said to be a quasi-isomorphism if
	\begin{enumerate}[(a)]
		\item $\alpha(\mathcal{D})$ is a cofinal subset of $\mathcal{D}$,
		\item given any integer $n$ there exists $D_n\in \mathcal{D}$ such that for all $D\geq D_n$, $ \phi_D: X_D \rightarrow Y_{\alpha(D)}$ is an open immersion and codimension of $Y_{\alpha(D)}-\phi_D(X_D)$ in $Y_{\alpha(D)}$ is greater than $n$, i.e for $D\gg 0$ the maps $\phi_D: X_D \rightarrow Y_{\alpha(D)}$ are open immersion and very close to being surjective.
	\end{enumerate}
\end{definition}

Now we recall an important theorem which talks about the quasi-isomorphism between the ind-varieties defined in Definition \ref{Div_ind-variety} and \ref{Quot_ind-variety}.

\begin{theorem}\label{T2}
	There is a natural quasi-isomorphism between the ind-varieties ${\mathbf{Div}}^{r,n}$ and ${\mathbf{Q}}^{r,-n}$.
\end{theorem}
\begin{proof}
	See \cite[Remark, page-647]{BGL}.  Infact, let $D$ be an effective divisor on $C$ of degree $d$. Let $(\mathcal{F},q)\in \Quot_{\mathcal{O}_C(D)^r}^{rd-n}$.  Then we have the following exact sequence.
	\begin{equation*}
	\xymatrix{ 0 \ar[r]& \Ker (q) \ar[r] &\mathcal{O}_C(D)^r\ar[r]^{q} &\mathcal{F}
		\ar[r] &0},
	\end{equation*}
	where $\deg(\mathcal{F})=rd-n$. Tensoring this with $\mathcal{O}_C(-D)$ we get,
	\begin{equation*}
	\xymatrix{ 0 \ar[r]& \Ker (q)\otimes \mathcal{O}_C(-D) \ar[r] &\mathcal{O}_C^r\ar[rr]^{q_1} &&\mathcal{F}\otimes \mathcal{O}_C(-D)
		\ar[r] &0},
	\end{equation*}
	where $\deg(\mathcal{F}\otimes \mathcal{O}_C(-D))=rd-n$.  Hence, $(\mathcal{F}\otimes \mathcal{O}_C(-D),q_1)\in \Quot_{\mathcal{O}_C^r}^{rd-n}.$  So we get a map $\Quot_{\mathcal{O}_C(D)^r}^{rd-n}\rightarrow \Quot_{\mathcal{O}_C^r}^{rd-n}$.  Restricting this map to the rational points of $\Quot_{\mathcal{O}_C(D)^r}^{rd-n}$, we obtain a map $\Div^{r,n}(D)\rightarrow Q^{r,-n}(D)$.  This map in turn will induce the required quasi-isomorphism 
	\begin{equation*}
	{\mathbf{Div}}^{r,n}\rightarrow {\mathbf{Q}}^{r,-n}.
	\end{equation*}  
\end{proof}
\begin{remark}\label{ind-variety of divisors}
By Theorem \ref{T2}, we can interpret ${\mathbf{Q}}^{r,-n}$ as the ind-variety of $(r,n)$-divisors on $C$.
\end{remark}
Now we are in a stage to describe what we mean by the diagonal property and the weak point property of an ind-variety.  In this regard, we have couple of definitions as follows.  The notion of smoothness of an ind-variety (cf. \cite[\S 2, p. 643]{BGL}) motivates us to define the following two notions relevant to our context. 
\begin{definition}\label{indvariety_dp}
Let $\Lambda$ be a filtered ordered set.  Let $X=\{X_{\lambda}, f_{\lambda \mu}\}_{\lambda, \mu \in \Lambda}$ be an ind-variety.  Then $X$ is said to have the diagonal property (respectively weak point property) if there exists some $\lambda_0\in \Lambda$ such that for all $\lambda\geq \lambda_0$, the varieties $X_{\lambda}$'s have the diagonal property (respectively weak point property).
\end{definition}

Let us now associate a rational number to a given higher rank divisor.  In fact, this number helps us to find some ind-varieties having the diagonal property and weak point property. 
\begin{definition}\label{slope of higher rank divisor}
For a given $(r,n)$-divisor, the rational number $\tfrac{n}{r}$ is said to its slope.
\end{definition}
 
We now prove a couple of theorems about the diagonal property and weak point property of ind-varieties of $(r,n)$-divisors, when the rational number as in Definition \ref{slope of higher rank divisor} is in fact an integer.
\begin{theorem}\label{integral slope case}
Let $C$ be a smooth projective curve over $\mathbb{C}$.  Also let $r\geq 1$ and $n$ be two integers.  Then the ind-variety of $(r,n)$-divisors having integral slope on $C$ has the weak point property. 
\end{theorem}
\begin{proof}
It can be noted that a $(r,n)$-divisor is of integral slope if and only if $n$ is an integral multiple of $r$, by Definition \ref{slope of higher rank divisor}.  Therefore, the ind-variety ${\mathbf{Div}}^{r,kr}$, or equivalently ${\mathbf{Q}}^{r,-kr}$ by Remark \ref{ind-variety of divisors}, is the ind-variety of higher rank divisors of integral slope.

Let $D$ be an effective divisor of degree $d$ on $C$.  Then we have, $Q^{r,-n}(D)=\Quot_{\mathcal{O}_C^r}^{rd-n}$ by Definition \ref{Quot schemes as constituent of ind-variety}.  Now if $n=rk$ for some integer $k$, then $Q^{r,-rk}(D)=\Quot_{\mathcal{O}_C^r}^{rd-rk}=\Quot_{\mathcal{O}_C^r}^{r(d-k)}$.  Now let's pick an effective divisor $D_0$ of degree $d_0$ satisfying the inequality $d_0>k$.  Then for all $D\geq D_0$ and $n=rk$, we have 
\begin{equation}\label{eqn_1_wpp}
\deg(D)\geq \deg(D_0)=d_0>k, 
\end{equation}
and 
\begin{equation}\label{eqn_2_wpp}
Q^{r,-n}(D)=Q^{r,-rk}(D)=\Quot_{\mathcal{O}_C^r}^{r(\deg(D)-k)}.
\end{equation}
Here $r(\deg(D)-k)$ is a positive integer by \eqref{eqn_1_wpp}.  Therefore, by Theorem \ref{Quot scheme_wpp}, Definition \ref{indvariety_dp} and \eqref{eqn_2_wpp}, the ind-variety ${\mathbf{Q}}^{r,-kr}$ has the weak point property.  Hence we have the assertion.  
\end{proof}
\begin{theorem}\label{rank one case}
Let $C$ be a smooth projective curve over $\mathbb{C}$ and $n$ any given integer. Then the ind-variety of $(1,n)$-divisors on $C$ has the diagonal property.
\end{theorem}
\begin{proof}
Let $D$ be an effective divisor of degree $d$ on $C$.  Then we have, $Q^{1,-n}(D)=\Quot_{\mathcal{O}_C}^{d-n}$.  Now let's pick an effective divisor $D_1$ of degree $d_1$ satisfying the inequality $d_1>n$.  Then for all $D\geq D_1$, we have 
\begin{equation}\label{eqn_1_dp}
\deg(D)\geq \deg(D_1)=d_1>n,
\end{equation}
and  
\begin{equation}\label{eqn_2_dp}
Q^{1,-n}(D)=\Quot_{\mathcal{O}_C}^{\deg(D)-n}\simeq \Sym^{\deg(D)-n}(C).
\end{equation}
Here $\deg(D)-n$ is a positive integer by \eqref{eqn_1_dp}.  Therefore, by Theorem \ref{symm prod_dp}, Definition \ref{indvariety_dp} and \eqref{eqn_2_dp}, the ind-variety ${\mathbf{Q}}^{1,-n}$ of all $(1,n)$-divisors has the diagonal property.  
\end{proof}
\begin{remark}
It can be noted a particular case of Theorem \ref{integral slope case}, namely the case $r=1$, follows from Theorem \ref{rank one case} and Remark \ref{dp implies wpp}.
\end{remark}
In \cite{PT}, Penkov and Tikhomirov studied the Barth-Van de Ven-Tyurin-Sato theorem on a locally complete linear ind-variety.  Given an ind-variety $\mathbf{X}=\{X_\lambda,f_{\lambda \mu}\}_{\lambda,\mu \in \Lambda}$, where the index set $\Lambda$ is any filtered ordered set, it is enough to check the locally completeness and linearity of an ind-variety for a countable linearly ordered subset $\mathcal{M}$ of $\Lambda$ such that for any $m,n,r \in \mathcal{M}$ with $m\leq n \leq r$ the diagram 
\begin{equation*}
\xymatrix{
X_{m} \ar[rd]_{f_{mr}}\ar[r]^{f_{mn}} & X_{n}
		\ar[d]^{f_{nr}}\\ 
		 & X_{r}}
\end{equation*}
commutes, as the ind-variety doesn't change after restricting the index set $\Lambda$ to any such $\mathcal{M}$.  In fact, in \cite{PT}, the authors defined a locally complete linear ind-variety by restricting the index set to such a countable subset $\mathcal{M}$ of a (possibly bigger) index set $\Lambda$.   

Here we show that the ind-variety of $(1,n)$-divisors, which satisfies diagonal property (cf. Theorem \ref{rank one case}), is in fact a locally complete linear ind-variety.  For that, we recall what one means by a locally complete linear ind-variety (cf. \cite[p. 816]{PT}).
\begin{definition}\label{locally complete ind-variety}
An ind-variety $\mathbf{X}=\{X_m,f_{mn}\}_{m,n \in \mathcal{N}}$ is locally complete if the varieties $X_m$'s are smooth complete algebraic varieties, $\lim_{m\rightarrow \infty} \dim X_m=\infty$ and the morphisms $f_{mn}:X_{m} \rightarrow X_{n}$ are embeddings.  
\end{definition} 
\begin{definition}\label{linear ind-variety}
A locally complete ind-variety $\mathbf{X}=\{X_m,f_{mn}\}_{m,n \in \mathcal{N}}$ is linear if the morphisms $f_{mn}^{\ast}:\Pic X_{n}\rightarrow \Pic X_{m}$ induced by the morphisms $f_{mn}$ on Picard groups are epimorphisms for almost all $m,n \in \mathcal{N}$.
\end{definition}

Let $J_d(C)$ be the moduli space of isomorphism classes of line bundles of degree $d$ over $C$ and $J(C)$ be the Jacobian variety of $C$.  Let us choose a point $P\in C$.   Consider the following composition map:
\begin{equation}\label{Abel-Jacobi map}
\begin{tikzcd}
\alpha_{d,P}:\Sym^d(C)\arrow[r,"\alpha_{d}"]& J_d(C)\arrow[r,"\otimes \mathcal{O}(-dP)"]& J(C)\\
D\arrow[r,mapsto]&\mathcal{O}_C(D)\arrow[r,mapsto]& \mathcal{O}_C(D-dP).
\end{tikzcd}
\end{equation}
We now check that the fibres of this map are projective spaces.  To be precise, we have the following lemma.
\begin{lemma}\label{using projective bundleness of symmetric products}
Let $C$ be a smooth projective curve of genus $g$.  Let $d$ be any positive integers satisfying $d\geq 2g-1$.  Then the fibre of the map $\alpha_{d,P}$, as in \eqref{Abel-Jacobi map}, over any $\mathcal{L}\in J(C)$ is $\mathbb{P}(H^0(C,\mathcal{L}\otimes dP))$ and hence is isomorphic to $\mathbb{P}^{d-g}$.
\end{lemma}
\begin{proof}
Let us denote the canonical line bundle over $C$ by $\omega_C$.  Then for any $D\in \Sym^d(C)$, the degree $\deg(\omega_C\otimes \mathcal{O}_C(D)^{\ast})$ of the line bundle $\omega_C\otimes \mathcal{O}_C(D)^{\ast}$ is $2g-2-d$.  Hence, for $d\geq 2g-1$, by Serre duality we have:
\begin{equation}\label{Serre duality}
h^1(C,\mathcal{O}_C(D))=h^0(C, \omega_C\otimes \mathcal{O}_C(D)^{\ast})=0.
\end{equation}
Therefore, by Riemann-Roch theorem and \eqref{Serre duality}, we have:
\begin{equation}\label{Riemann Roch theorem}
\begin{split}
h^0(C,\mathcal{O}_C(D))&=h^0(C,\mathcal{O}_C(D))-h^1(C,\mathcal{O}_C(D))\\
&=\deg(\mathcal{O}_C(D))+(1-g)\\
&=d-g+1.
\end{split}
\end{equation}
By Abel's theorem (cf. \cite[p. 18]{ACGH}), fibre of the map $\alpha_d$, as in \eqref{Abel-Jacobi map}, over any line bundle $L\in J_d(C)$
is the complete linear system $|D|$ of a divisor $D$ on $C$ with $\mathcal{O}_C(D)=L$.  Moreover, we have: 
\begin{equation}\label{complete linear system as projective space}
|D|=\mathbb{P}(H^0(C,\mathcal{O}_C(D))).
\end{equation}
Therefore, by \eqref{Riemann Roch theorem} and \eqref{complete linear system as projective space}, we obtain that the map $\alpha_d$, as in \eqref{Abel-Jacobi map} is a projective bundle, with fibres $\mathbb{P}^{d-g}$, for all $d\geq 2g-1$.  Moreover, so is the map $\alpha_{d,P}$ as the map $\otimes\mathcal{O}_C(-dP)$ is an isomorphism between $J_d(C)$ and $J(C)$ (cf. \eqref{Abel-Jacobi map with fibers in diagram}).
\begin{equation}\label{Abel-Jacobi map with fibers in diagram}
\begin{tikzcd}
\mathbb{P}^{d-g}\arrow[d, mapsto]\arrow[r, hook]  &\Sym^{d}(C)\arrow[d,"\alpha_d"'] \arrow[drr,"\alpha_{d,P}"]&&\\
\{L\}\arrow[r, hook] &J_{d}(C)\arrow[rr,"\simeq","\otimes \mathcal{O}_C(-dP)"']&&J(C)
\end{tikzcd}
\end{equation}
In fact, from \eqref{complete linear system as projective space} it follows that the fibre over any $\mathcal{L}\in J(C)$ of the map $\alpha_{d,P}$ is $\mathbb{P}(H^0(C,\mathcal{L}\otimes dP))$. 
\end{proof}
We now check locally completeness and linearity of the ind-variety of $(1,n)$-divisors.
\begin{theorem}\label{locally complete and linear_rank one case}
Let $C$ be a smooth projective curve over $\mathbb{C}$ and $n$ any given integer. Then the ind-variety of $(1,n)$-divisors on $C$ is a locally complete linear ind-variety.
\end{theorem}
\begin{proof}
The ind-variety of $(1,n)$-divisors on $C$ is ${\mathbf{Q}}^{1,-n}$.  We now choose the index set $\mathcal{N}$ to be the set $\{mP\mid m\geq n\}$.  Then the ind-variety formed by the varieties $\{Q^{1,-n}(D)\}_{D\in \mathcal{N}}$ along with the corresponding morphisms as in \eqref{E7} is nothing but the ind-variety ${\mathbf{Q}}^{1,-n}$.  

Let $r$ and $m$ be two integers with $r\geq m \geq n$.  Let $(\mathcal{F}_m,q_m) \in \Quot^{m-n}_{\mathcal{O}_C}$.  Therefore we have the following exact sequence:
\begin{equation*}
\xymatrix{ 0 \ar[r]& \Ker (q_m) \ar[r] &\mathcal{O}_C\ar[r]^{q_m} &\mathcal{F}_m
	\ar[r] &0},
\end{equation*}
where degree of $\mathcal{F}_m$ is $m-n$ and hence degree of $\Ker (q_m)$ is $n-m$.  We then consider the morphism, as in \eqref{E7}, in this context : 
\begin{equation*}
\begin{split}
\mathcal{O}_C((m-r)P): \Quot^{m-n}_{\mathcal{O}_C} &\rightarrow \Quot^{r-n}_{\mathcal{O}_C}\\
(\mathcal{F}_m,q_m)&\mapsto (\mathcal{F}_r,q_r),
\end{split}
\end{equation*}
where $(\mathcal{F}_r,q_r)$ satisfies the following exact sequence:
\begin{equation*}
\xymatrix{ 0 \ar[r]& \Ker (q_m)\otimes \mathcal{O}_C((m-r)P) \ar[r] &\mathcal{O}_C\ar[r]^{q_r} &\mathcal{F}_r
	\ar[r] &0}.
\end{equation*}
Also, Let $\Ker (q_m)=\mathcal{O}_C(D)$, for a divisor $D$ of $C$ of degree $n-m$.  Then we have the following isomorphism :
\begin{equation*}
\begin{split}
\eta_m:\Quot^{m-n}_{\mathcal{O}_C} &\rightarrow \Sym^{m-n}(C)\\
(\mathcal{F}_m,q_m)&\mapsto \mathcal{O}_C(-D).
\end{split}
\end{equation*}
Therefore we have the following commutative diagram :
\begin{equation*}
\begin{tikzcd}
\Quot^{m-n}_{\mathcal{O}_C}\arrow[d,"\mathcal{O}_C((m-r)P)"']\arrow[rr, "\eta_m", "\simeq"']&& \Sym^{m-n}(C)\arrow[d,"\psi_{mr}"]\\
\Quot^{r-n}_{\mathcal{O}_C}\arrow[rr, "\eta_r"', "\simeq"]&& \Sym^{r-n}(C)\;,
\end{tikzcd}
\end{equation*}
where the map $\psi_{mr}$ is given as follows :
\begin{equation}\label{inclusions of symmetric products}
\begin{split}
\psi_{mr}:\Sym^{m-n}(C)&\rightarrow \Sym^{r-n}(C)\\
D&\mapsto D+(r-m)P.
\end{split}
\end{equation}
Therefore, the ind-variety ${\mathbf{Q}}^{1,-n}$ is nothing but $\{\Sym^{m-n}(C),\psi_{mr}\}_{m,r\in \mathcal{N}}$.  

Now $\Sym^{m-n}(C)$ is smooth as $C$ is a smooth curve (cf. \cite[Proposition 3.2, p.~9]{Mi2}).  Also, as $C$ is projective, so is $C^{m-n}$.  So, $\Sym^{m-n}(C)$, being a quotient of $C^{m-n}$ by a finite map, is also a projective variety and hence complete.  Therefore, the morphisms $\psi_{mr}$ are embeddings of complete varieties. 

Moreover as, $\dim \Quot_{\mathcal{O}_C}^{m-n}=\dim \Sym^{m-n}(C)=m-n$, 
$$\lim_{m\rightarrow \infty} \dim Q^{1,-n}(mP)=\lim_{m\rightarrow \infty} (m-n)=\infty.$$

Therefore,  ${\mathbf{Q}}^{1,-n}$ is a locally complete ind-variety.
	
Now, for any $r>2g-2+n$, $\Sym^{r-n}(C)$ is a projective bundle over $J(C)$ by Lemma \ref{using projective bundleness of symmetric products}.  Therefore, by \cite[Chapter \RomanNumeralCaps{2}, Exercise 7.9, p.~170]{Ha},
\begin{equation}\label{Picard group structure}
\Pic(Q^{1,-n}(rP))=\Pic(\Sym^{r-n}(C))=\Pic(J(C))\oplus \mathbb{Z}.
\end{equation}	
Moreover, for any $r\geq m>2g-2+n$, we have the following commutative diagram, where the maps $\psi_{mr}$, $\alpha_{m-n,P}$ and $\alpha_{r-n,P}$ are as in \eqref{inclusions of symmetric products} and \eqref{Abel-Jacobi map} respectively :
\begin{equation*}
\begin{tikzcd}
\Sym^{m-n}(C)\arrow[rr, "\psi_{mr}"]\arrow[rd,"\alpha_{m-n,P}"'] &&\Sym^{r-n}(C)\arrow[ld,"\alpha_{r-n,P}"]\\
&J(C)&
\end{tikzcd}
\end{equation*}
Therefore, we conclude that the $\Pic(J(C))$ component of $\Pic(\Sym^{r-n}(C))$, as in \eqref{Picard group structure}, is pulled back to $\Pic(J(C))$ part of $\Pic(\Sym^{m-n}(C))$ via $\psi_{mr}^{\ast}$.  Furthermore, by Lemma \ref{using projective bundleness of symmetric products}, we have the following commutative diagram for any $\mathcal{L}\in J(C)$ : 
\begin{equation*}
\begin{tikzcd}
\mathbb{P}(H^0(C,\mathcal{L}\otimes (m-n)P))\arrow[d, hook]\arrow[rr, hook]&& \mathbb{P}(H^0(C,\mathcal{L}\otimes (r-n)P))\arrow[d, hook]\\
\Sym^{m-n}(C)\arrow[rr, hook, "\psi_{mr}"]&&\Sym^{r-n}(C)
\end{tikzcd}
\end{equation*}
Therefore, $\mathcal{O}(1)_{\mathbb{P}(H^0(C,\mathcal{L}\otimes (r-n)P))}$, a generator of $\mathbb{Z}$ component of $\Pic(\Sym^{r-n}(C))$, as in \eqref{Picard group structure}, is restricted to $\mathcal{O}(1)_{\mathbb{P}(H^0(C,\mathcal{L}\otimes (m-n)P))}$, which in turn generates $\mathbb{Z}$ component of $\Pic(\Sym^{m-n}(C))$.  Altogether, for any $r\geq m>2g-2+n$ we have,
\begin{equation}\label{pullback of Picard group}
\begin{split}
\psi_{mr}^{\ast}(\Pic(\Sym^{r-n}(C)))&=\psi_{mr}^{\ast}(\Pic(J(C))\oplus \mathbb{Z})\\&=\Pic(J(C))\oplus \mathbb{Z}=\Pic(\Sym^{m-n}(C)).
\end{split}
\end{equation} 
Hence, ${\mathbf{Q}}^{1,-n}$ is a linear ind-variety as well.
\end{proof}
\begin{corollary}
Let $C$ be a smooth projective curve over $\mathbb{C}$ and $n$ any given integer. Then the Picard group of the ind-variety of $(1,n)$-divisors on $C$ is $\Pic(J(C))\oplus \mathbb{Z}$.
\end{corollary}
\begin{proof}
The Picard group $\Pic({\mathbf{Q}}^{1,-n})$ of the ind-variety ${\mathbf{Q}}^{1,-n}$ of $(1,n)$-divisors on $C$ is
defined by $\varprojlim \Pic(Q^{1,-n}(D))$, (cf. \cite[p.~816]{PT}).  Therefore, by \eqref{pullback of Picard group}, we have :
\begin{equation*}
\begin{split}
\Pic({\mathbf{Q}}^{1,-n})&=\varprojlim \Pic(Q^{1,-n}(D))\\
&=\varprojlim_{\Pic(\Sym^{r-n}(C))\in \mathcal{F}}\Pic(\Sym^{r-n}(C))=\Pic(J(C))\oplus \mathbb{Z},
\end{split}
\end{equation*}
where $\mathcal{F}$ is nothing but $\{\Pic(\Sym^{m-n}(C)),\psi_{mr}^{\ast}:\Pic(\Sym^{r-n}(C))\rightarrow \Pic(\Sym^{m-n}(C))\}_{r,n\in \mathcal{N}}$, as in Theorem \ref{locally complete and linear_rank one case}.
\end{proof}
\section{The diagonal property of the Hilbert schemes associated to a constant polynomial and its good partitions}\label{sec: Hilbert scheme}
In this section, we talk about the Hilbert schemes of a curve associated to a polynomial and its good partitions.     First we mention the importance of studying such Hilbert schemes and then show that few of these Hilbert schemes satisfy the diagonal property.  Moreover, we provide an upper bound on the number of such Hilbert schemes.

Let $P(t)$ be a polynomial with rational coefficients.  We use the notation $\Hilb^{P}_C$ to denote the Hilbert scheme parametrizing all subschemes of $C$ having Hilbert polynomial $P(t)$.  Let $n$ be a positive integer.  Then interpreting $n$ as a constant polynomial, by $\Hilb^n_C$ we mean the Hilbert scheme parametrizing subschemes of $C$ having Hilbert polynomial $n$.  Let us recall the notion of a good partition of a polynomial and a Hilbert scheme associated to that. 
\begin{definition}\label{gp_defn}
	Let $\underline{P}=(P_i)_{i=1}^s$ be a family of polynomials with rational coefficients.  Then $\underline{P}$ is said to be a good partition of $P$ if $\sum_{i=1}^s P_i=P$ and $\Hilb_C^{P_i}\neq \phi$ for all $i$.
\end{definition} 
\begin{definition}
	The Hilbert scheme associated to a polynomial $P$ and its good partition $\underline{P}$, denoted by $\Hilb_C^{\underline{P}}$ , is defined as $\Hilb^{\underline{P}}_C:=\Hilb^{P_1}_C\times_\mathbb{C}\cdots \times_\mathbb{C} \Hilb^{P_s}_C$.
\end{definition}
\begin{remark}\label{motivation for checking dp for Hilbert schemes}
At this point it is worthwhile to mention the importance of the Hilbert scheme $\Hilb^{\underline{P}}_C$.  Recall that by  $\Quot^{P}_{\mathcal{F}}$ we denote the Quot scheme parametrizing all torsion quotients of $\mathcal{F}$ having having Hilbert polynomial $P(t)$.  We have a decomposition of $\Quot_{\mathcal{O}_C^r}^P$ as follows, whenever $\Quot_{\mathcal{O}_C^r}^P$ is smooth.
	\begin{equation*}\label{E8}
	\Quot_{\mathcal{O}_C^r}^P=\bigsqcup_{\substack{\underline{P}\; such\; that \;\underline{P}\\is \;a\; good\; partition\; of\; P }}\mathcal{S}_{\underline{P}},
	\end{equation*}
	where each $\mathcal{S}_{\underline{P}}$ is smooth, the torus $\mathbb{G}_m^r$-invariant, locally closed and isomorphic to a vector bundle over the scheme $\Hilb^{\underline{P}}_C$, (cf. \cite[p. 610]{B}).  Therefore, the cohomology of $\Quot_{\mathcal{O}_C^r}^P$ can be given by the direct sum of the cohomologies of $\Hilb^{\underline{P}}_C$, where the sum varies over the good partitions of the polynomial $P$.  So to study the cohomology ring $H^{\ast}(\Quot_{\mathcal{O}_C^r}^P)$, it is enough the cohomology rings $H^{\ast}(\Hilb^{\underline{P}}_C)$, $\underline{P}$ being good partition of the polynomial $P$.  Now, to get hold of the cohomology rings $H^{\ast}(\Hilb^{\underline{P}}_C)$, it's nice to get hold of the structure of the Hilbert scheme $\Hilb^{\underline{P}}_C$.  Now, as the diagonal property and the weak point property force strong conditions on the underlying variety (cf. \cite{PSP}), therefore to the study the cohomology of $\Quot_{\mathcal{O}_C^r}^P$ it's reasonable enough to check whether the Hilbert schemes $\Hilb^{\underline{P}}_C$'s posses these properties or not. 
\end{remark}
Remark \ref{motivation for checking dp for Hilbert schemes} motivates us to talk about the diagonal property of the Hilbert schemes associated to a constant polynomial and some particular good partitions of the same.  Towards that, we have the following lemma followed by a definition. 
\begin{definition}\label{partition of an integer} 
Let $n$ be a positive integer.  Then a partition of $n$ of length $s$ is given by a $s$-tuple $(n_1,n_2,\ldots,n_s)$ such that $\sum_{i=1}^r n_i=n$ and $n_1\geq n_2\geq \cdots \geq n_s>0$ for all $i$.  The integers $n_i$'s are called parts of the partition $(n_1,n_2,\ldots,n_s)$.
\end{definition}
\begin{lemma}\label{partition_a good partition}
Let $n$ be a given positive integer.  Then any partition of $n$ is also a good partition of $n$ and vice versa.
\end{lemma}
\begin{proof}
Let $n$ be a positive integer.  As $n_i>0$ and $\Hilb^{n_i}_C$ is isomorphic to the moduli space $\Sym^{n_i}(C)$ of effective divisors of degree $n_i$ over $C$, we have $\Hilb^{n_i}_C\neq \emptyset$ for all $i$.  Therefore, by Definition \ref{gp_defn}, the chosen partition of $n$ is a good partition as well.   Converse part follows from the fact that given any given integer $k$, $\Hilb^{k}_C$ is non-empty if and only if $k$ is positive.  
\end{proof} 
 
We now have the following lemma which says that if two varieties have diagonal property, then so does their product.  The statement of the lemma can be found in the literature (cf. \cite[p. 1235]{PSP}, \cite[p. 47]{D}).  The proof though is not available to the best of our knowledge.  Therefore, for the sake of completeness, we provide the proof of the same.   
\begin{lemma}\label{dp for product}
Let $X_1$ and $X_2$ be two varieties over $\mathbb{C}$ satisfying the the diagonal property.  Then the product variety $X_1\times X_2$ also have the diagonal property. 
\end{lemma} 
\begin{proof}
Let $i=1,2$.  Let the dimension of $X_i$ be $n_i$.  As $X_i$ satisfy the diagonal property, by Definition \ref{dp}, there exists a vector bundle $E_i$ over $X_i\times X_i$ of rank $n_i$ and a section $s_i$ of $E_i$ such that the zero scheme $Z(s_i)$ of $s_i$ is the diagonal $\Delta_{X_i}$.  Let $p_i:(X_1\times X_2)\times (X_1\times X_2) \rightarrow X_i\times X_i$ are the projection maps given by $p_i((x_1,x_2,x'_1,x'_2))=(x_i,x'_i)$.  Consider the vector bundle $p_1^{\ast}E_1\oplus p_2^{\ast}E_2$ of rank $n_1+n_2$ over $(X_1\times X_2)\times (X_1\times X_2)$.  Then the zero scheme $Z((p_1^{\ast}s_1, p_2^{\ast}s_2))$ of the section $(p_1^{\ast}s_1, p_2^{\ast}s_2)$ of $p_1^{\ast}E_1\oplus p_2^{\ast}E_2$ is the diagonal $\Delta_{X_1\times X_2}$ of $X_1\times X_2$.  Hence, the assertion follows by Definition \ref{dp}.
\end{proof} 

By multiprojective space, one means product of projective spaces (cf. \cite[Chapter 3, \S 36, p. 150]{Mus}).  The following lemma says that given two distinct partition of a positive integer $n$, the corresponding multiprojective spaces are not isomorphic.  More precisely, we have the following.
\begin{proposition}\label{non isomorphic multiprojective spaces}
Let $n$ be a positive integer.  Let $(m_1,m_2,\ldots,m_s)$ and $(n_1,n_2,\ldots,n_s)$ be two distinct partitions of $n$ of same length $s$.  Then $\mathbb{P}^{m_1}\times 
\mathbb{P}^{m_2}\times \cdots \times \mathbb{P}^{m_s}$ is not isomorphic to  $\mathbb{P}^{n_1}\times \mathbb{P}^{n_2}\times \cdots \times \mathbb{P}^{n_s}$.
\end{proposition}
\begin{proof}
The cohomology ring $H^{\ast}(\mathbb{P}^{m_i},\mathbb{Z})$ of $\mathbb{P}^{m_i}$ is the ring $\tfrac{\mathbb{Z}[x_i]}{\langle x_i^{m_i+1} \rangle}$, where $x_i$ is a generator of $H^2(\mathbb{P}^{m_i},\mathbb{Z})(\simeq \mathbb{Z})$, for all $1\leq i\leq s$.  Similarly,  $H^{\ast}(\mathbb{P}^{n_i},\mathbb{Z})=\tfrac{\mathbb{Z}[y_i]}{\langle y_i^{n_i+1} \rangle}$, for all $1\leq i\leq s$.  Now, by K\"unneth formula, we have 
\begin{equation*}\label{cohomology of multiprojective space}
\begin{split}
H^{\ast}(\mathbb{P}^{m_1}\times 
\mathbb{P}^{m_2}\times \cdots \times \mathbb{P}^{m_s},&\mathbb{Z})=\dfrac{\mathbb{Z}[x_1]}{\langle x_1^{m_1+1} \rangle}\otimes \dfrac{\mathbb{Z}[x_2]}{\langle x_2^{m_2+1} \rangle}\otimes \cdots \otimes \dfrac{\mathbb{Z}[x_s]}{\langle x_s^{m_s+1} \rangle}\\
&=\dfrac{\mathbb{Z}[x_1,x_2,\cdots, x_s]}{\langle x_1^{m_1+1}, x_2^{m_2+1}, \cdots ,x_s^{m_s+1} \rangle}=M\,(\text{say}).
\end{split}
\end{equation*} 
Similarly, we have $H^{\ast}(\mathbb{P}^{n_1}\times 
\mathbb{P}^{n_2}\times \cdots \times \mathbb{P}^{n_s},\mathbb{Z})=\tfrac{\mathbb{Z}[y_1,y_2,\cdots, y_s]}{\langle y_1^{n_1+1}, y_2^{n_2+1}, \cdots ,y_s^{n_s+1} \rangle}=N\,(\text{say})$.
Now, for all $1\leq i\leq s$, let 
\begin{equation*}
pr_{m_i}: \mathbb{P}^{m_1}\times 
\mathbb{P}^{m_2}\times \cdots \times \mathbb{P}^{m_s} \longrightarrow \mathbb{P}^{m_i}
\end{equation*}
be the $i$-th projection map.  Then, in the ring $M$, $x_i$ can be interpreted as the first Chern class $c_1(pr_{m_i}^{\ast}(\mathcal{O}_{\mathbb{P}^{m_i}}(1)))$ of the pullback of the hyperplane bundle $\mathcal{O}_{\mathbb{P}^{m_i}}(1)$ on $\mathbb{P}^{m_i}$ via the projection map $pr_{m_i}$.  Moreover, under the identification of the Picard group $\Pic(\mathbb{P}^{m_1}\times 
\mathbb{P}^{m_2}\times \cdots \times \mathbb{P}^{m_s})$ of $\mathbb{P}^{m_1}\times 
\mathbb{P}^{m_2}\times \cdots \times \mathbb{P}^{m_s}$ with direct sum of $s$ copies of $\mathbb{Z}$, the line bundle $pr_{m_i}^{\ast}(\mathcal{O}_{\mathbb{P}^{m_i}}(1))$ is nothing but $(0,0,\ldots,1,\dots,0)$, $1$ being in the $i$-th place.  Therefore, $pr_{m_i}^{\ast}(\mathcal{O}_{\mathbb{P}^{m_i}}(1))$ is globally generated (cf. \cite[Theorem 7.1, p. 150]{Ha}).  As $\mathbb{P}^{m_1}\times \mathbb{P}^{m_2}\times \cdots \times \mathbb{P}^{m_s}$, being a projective variety, is complete and therefore $pr_{m_i}^{\ast}(\mathcal{O}_{\mathbb{P}^{m_i}}(1))$ is nef (cf. \cite[Example 1.4.5, p. 51]{La}), but not ample (cf. \cite[Example 7.6.2, p. 156]{Ha}), for all $1\leq i\leq s$.  Following similar notations, $y_i=c_1(pr_{n_i}^{\ast}(\mathcal{O}_{\mathbb{P}^{n_i}}(1)))$, where the bundle $pr_{n_i}^{\ast}(\mathcal{O}_{\mathbb{P}^{n_i}}(1))$ is nef, but not ample.  Furthermore, the extremal rays of the nef cones of $\mathbb{P}^{m_1}\times 
\mathbb{P}^{m_2}\times \cdots \times \mathbb{P}^{m_s}$ and $\mathbb{P}^{n_1}\times 
\mathbb{P}^{n_2}\times \cdots \times \mathbb{P}^{n_s}$ are the one-dimensional sub cones generated by $x_i$'s and $y_i$'s respectively and these are the primitive generators as well.  So, if there exists an isomorphism $\Psi$ from $\mathbb{P}^{m_1}\times \mathbb{P}^{m_2}\times \cdots \times \mathbb{P}^{m_s}$ to $\mathbb{P}^{n_1}\times \mathbb{P}^{n_2}\times \cdots \times \mathbb{P}^{n_s}$, then $y_i$ should map to some $x_j$ under the induced isomorphism $\Psi^{\ast}$ at the cohomology level, that is to say $(\Psi^{\ast}y_1,\Psi^{\ast}y_2, \cdots, \Psi^{\ast}y_s)=(x_{\sigma(1)},x_{\sigma(2)}, \cdots, x_{\sigma(s)})$, for some $\sigma \in S_s$.  Now, as the partitions $(m_1,m_2,\ldots,m_s)$ and $(n_1,n_2,\ldots,n_s)$ of $n$ are distinct, $m_i\neq n_i$ for some $i$.  So, the rings $M$ and $N$ are not isomorphic.  But that is a contradiction to our assumption that $\mathbb{P}^{m_1}\times 
\mathbb{P}^{m_2}\times \cdots \times \mathbb{P}^{m_s}$ and $\mathbb{P}^{n_1}\times \mathbb{P}^{n_2}\times \cdots \times \mathbb{P}^{n_s}$ are isomorphic.  Hence, the assertion follows.
\end{proof}

We now generalise Proposition \ref{non isomorphic multiprojective spaces} from projective spaces to symmetric product of curves.  Towards that we have the following definition.

\begin{definition}\label{multi_symmetric_product}
Let $C$ be a smooth projective curve over $\mathbb{C}$ and $(n_1,n_2,\ldots,n_r)$ a partition of a given positive integer $n$.  Then a multi symmetric product of $C$ of type $[(n_1,n_2,\ldots,n_r),n]$ is defined as the product $\Sym^{n_1}(C)\times \Sym^{n_2}(C)\times \cdots \times \Sym^{n_r}(C)$.
\end{definition}

\begin{proposition}\label{Betti numbers of symmetric product}
Let $C$ be a smooth projective curve over $\mathbb{C}$ of genus $g$.  Then the $r$-th Betti number $B_r$ of $\Sym^n(C)$ is given by 
\begin{equation*}
B_r=B_{2n-r}=\big(\begin{smallmatrix}
  2g\\
  r
\end{smallmatrix}\big)+\big(\begin{smallmatrix}
  2g\\
  r-2
\end{smallmatrix}\big)+\ldots \;\;, 
\end{equation*}
where $0\leq r\leq n$.
\end{proposition}
\begin{proof}
See \cite[Equation 4.2, p.~322]{Mac}.
\end{proof}

\begin{lemma}\label{first Betti number of multi symmetric product}
Let $C$ be a smooth projective curve over $\mathbb{C}$ of genus $g$.  Then the first Betti number $B_1$ of the multi symmetric product of $C$ of type $[(n_1,n_2,\ldots,n_r),n]$ is $2rg$.
\end{lemma}
\begin{proof}
We prove this by induction on $r$, the length of a partition.  For $r=1$, by Proposition \ref{Betti numbers of symmetric product}, we have $B_1=\big(\begin{smallmatrix}
  2g\\
  1
\end{smallmatrix}\big)=2g$. Let us now consider the case $r=2$.  Let us denote the Poincar\'e polynomial of a space $X$ by $\text{P.P\; of \;}X$.  Then we have:
\begin{equation*}
\begin{split}
\text{P.P\; of \; } \Sym^{n_1}(C)&\;\text{is}\;1+2gx+\cdots+x^{2n_1},\\
\text{P.P\; of \; } \Sym^{n_2}(C)&\;\text{is}\;1+2gx+\cdots+x^{2n_2}.
\end{split}
\end{equation*}  
Therefore, as Poincar\'e polynomial of a product space is obtained by the product of the Poincar\'e polynomials of the corresponding spaces, we have:
\begin{equation*}
\text{P.P\; of \; } \Sym^{n_1}(C)\times \Sym^{n_2}(C)\;\text{is}\; 1+4gx+\cdots+x^{2(n_1+n_2)}.
\end{equation*}
Therefore, the first Betti number of $\Sym^{n_1}(C)\times \Sym^{n_2}(C)$ is $4g$.\\
Let us now assume that for $r=m$, the first Betti number of the multi symmetric product of $C$ of type $[(n_1,n_2,\ldots,n_m),n]$ is $2mg$.  Then the Poincar\'e polynomial of  the multi symmetric product of $C$ of type $[(n_1,n_2,\ldots,n_{m+1}),n]$ is given by:
\begin{equation*}
\text{P.P\; of \; } \Sym^{n_1}(C) \times \Sym^{n_2}(C)\times \cdots \times \Sym^{n_{m+1}}(C),
\end{equation*}
which is same as 
\begin{equation*}
\text{P.P\; of \; } \Sym^{n_1}(C)\times \Sym^{n_2}(C)\times \cdots \times \Sym^{n_m}(C)\cdot \text{P.P\; of \; } \Sym^{n_{m+1}}(C),
\end{equation*}
which in turn is the product 
\begin{equation*}
(1+2mgx+\cdots + x^{2(n_1+n_2+\cdots+n_m)})\cdot (1+2gx+\cdots+x^{2n_{m+1}}),
\end{equation*}
which is nothing but
\begin{equation*}
1+2(m+1)gx+\cdots +x^{2(n_1+n_2+\cdots+n_{m+1})}.
\end{equation*}
Therefore, we obtain that the first Betti number of the multi symmetric product of $C$ of type $[(n_1,n_2,\ldots,n_{m+1}),n]$ is $2(m+1)g$.  Hence we have the assertion.
\end{proof}

\begin{proposition}\label{non isomorphic multi symmetric products for different lengths}
Let $C$ be a smooth projective curve over $\mathbb{C}$ of genus $g$ with $g\geq 1$.  Let $n$ be a positive integer, and $(n_1,n_2,\ldots,n_r)$ and $(m_1,m_2,\ldots,m_s)$ two distinct partitions of $n$ of different lengths.  Then the multi symmetric product of $C$ of type $[(n_1,n_2,\ldots,n_r),n]$ and $[(m_1,m_2,\ldots,m_s),n]$ are not isomorphic.
\end{proposition}
\begin{proof}
By Lemma \ref{first Betti number of multi symmetric product}, the first Betti number of the multi symmetric products of type $[(n_1,n_2,\ldots,n_r),n]$ and $[(m_1,m_2,\ldots,m_s),n]$ are $2rg$ and $2sg$ respectively.  Now as the partitions $(n_1,n_2,\ldots,n_r)$ and $(m_1,m_2,\ldots,m_s)$ are of different lengths, that is $r\neq s$, and $g>0$, the first Betti number of the multi symmetric product of corresponding types are also different, and hence they are not isomorphic. 
\end{proof}

We now have the following theorem which says about the diagonal property of the Hilbert schemes associated to a constant polynomial and its good partitions.  Moreover, we provide an upper bound on the number of such Hilbert schemes and prove that the obtained bound is the best possible bound.

\begin{theorem}\label{Main Theorem 3}
Let $C$ be a smooth projective curve over $\mathbb{C}$ and $n$ a positive integer.  Let $p(n)$ denote the number of partitions of $n$.  Then the following hold:
\begin{enumerate}
\item There are at most $p(n)$ many Hilbert schemes $\Hilb^{\underline{n}}_C$ (up to isomorphism) associated to the constant polynomial $n$ and its good partitions $\underline{n}$ satisfying diagonal property.
\item Moreover, this upper bound is is attained by any genus $0$ curve $C$ and hence is sharp.
\item Furthermore, for $n=1,2,3$, the upper bound is attained by any curve $C$.
\end{enumerate}
\end{theorem}
\begin{proof}
Given a positive integer $n$, a partition $(n_1,n_2,\ldots,n_s)$ of $n$ is also a good partition by Lemma \ref{partition_a good partition}. Moreover, the associated Hilbert scheme is given by $\Hilb^{n_1}_C\times_\mathbb{C} \Hilb^{n_2}_C\times_\mathbb{C} \cdots \times_\mathbb{C} \Hilb^{n_s}_C$.  As $\Hilb^m_C\simeq \Sym^m(C)$ for any positive integer $m$, by Theorem \ref{symm prod_dp} and Lemma \ref{dp for product}, we get that the associated Hilbert scheme $\Hilb^{n_1}_C\times_\mathbb{C} \Hilb^{n_2}_C\times_\mathbb{C} \cdots \times_\mathbb{C} \Hilb^{n_s}_C$ satisfies the diagonal property.  So, up to isomorphism, there could be at most as many such Hilbert schemes as there are partitions of $n$.  Hence the first part of the assertion follows from Lemma \ref{partition_a good partition}.

We now show that the obtained upper bound for number of Hilbert schemes associated to the good partitions of the constant polynomial $n$ satisfying diagonal property is in fact is achieved in genus $0$ case.  Indeed, let us consider $C=\mathbb{P}^1$.  Then, $\Hilb^n_{\mathbb{P}^1}=\Sym^n(\mathbb{P}^1)=\mathbb{P}^n$.  Now let us take two distinct partition of $n$, say $(n_1,n_2,\ldots,n_s)$ and $(n^{'}_1,n^{'}_2,\ldots,n^{'}_t)$.  Then, as before, we have the following two mutually exclusive and exhaustive cases:\\
\textit{First Case :} $s\neq t$\\
In this case, the associated Hilbert schemes $\Hilb^{n_1}_{\mathbb{P}^1}\times_\mathbb{C} \Hilb^{n_2}_{\mathbb{P}^1}\times_\mathbb{C} \cdots \times_\mathbb{C} \Hilb^{n_s}_{\mathbb{P}^1}$ and $\Hilb^{n^{'}_1}_{\mathbb{P}^1}\times_\mathbb{C} \Hilb^{n^{'}_2}_{\mathbb{P}^1}\times_\mathbb{C} \cdots \times_\mathbb{C} \Hilb^{n^{'}_t}_{\mathbb{P}^1}$ are not isomorphic as their Picard groups are not so.  That is to say,
\begin{equation*}
\begin{split}
\Pic(\Hilb^{n_1}_{\mathbb{P}^1}\times_\mathbb{C} \Hilb^{n_2}_{\mathbb{P}^1}\times_\mathbb{C} \cdots \times_\mathbb{C} \Hilb^{n_s}_{\mathbb{P}^1})&\simeq \oplus_{i=1}^s\mathbb{Z}\ncong \oplus_{i=1}^t\mathbb{Z}\\
&=\Pic(\Hilb^{n^{'}_1}_{\mathbb{P}^1}\times_\mathbb{C} \Hilb^{n^{'}_2}_{\mathbb{P}^1}\times_\mathbb{C} \cdots \times_\mathbb{C} \Hilb^{n^{'}_t}_{\mathbb{P}^1}).
\end{split}
\end{equation*}
\textit{Second Case :} $s=t$ \\
In this case, the associated Hilbert schemes $\Hilb^{n_1}_{\mathbb{P}^1}\times_\mathbb{C} \Hilb^{n_2}_{\mathbb{P}^1}\times_\mathbb{C} \cdots \times_\mathbb{C} \Hilb^{n_s}_{\mathbb{P}^1}$ and $\Hilb^{n^{'}_1}_{\mathbb{P}^1}\times_\mathbb{C} \Hilb^{n^{'}_2}_{\mathbb{P}^1}\times_\mathbb{C} \cdots \times_\mathbb{C} \Hilb^{n^{'}_s}_{\mathbb{P}^1}$ are not isomorphic by Proposition \ref{non isomorphic multiprojective spaces}.   

Hence the second part of the assertion follows. 

Last part of the assertion follows from Proposition \ref{non isomorphic multi symmetric products for different lengths} modulo the fact that no two partitions of $m$ are of same length for $m=1,2,3$.
\end{proof}

Let us now check whether multi symmetric product spaces corresponding to two different partitions of a given integer having same length are isomorphic or not.  Proposition \ref{non isomorphic multiprojective spaces} gives us the hope to believe that they too are non-isomorphic.  Let us calculate some Betti numbers of multi symmetric products corresponding to smaller values of $n$. Also we confine ourselves to the cases where the number of parts is less than equal to $3$ to avoid tedious calculations.

Let us take $n=5$, and the partitions $(4,1)$ and $(3,2)$.  Then by Proposition \ref{Betti numbers of symmetric product}, we have:
\begin{equation*}
\begin{split}
&\text{P.P\; of \; } \Sym^4(C)\;\text{is}\; 1+2gx+(1+\big(\begin{smallmatrix}
  2g\\
  2
\end{smallmatrix}\big))x^2+\cdots+x^8,\\
&\text{P.P\; of \; } \Sym^1(C)\;\text{is}\; 1+2gx+x^2.
\end{split}
\end{equation*}
Therefore,
\begin{equation}\label{PP of (4,1)}
\text{P.P\; of \; } \Sym^4(C)\times \Sym^1(C)\;\text{is}\; 1+4gx+(2+4g^2+\big(\begin{smallmatrix}
  2g\\
  2
\end{smallmatrix}\big))x^2+\cdots+x^{10}.
\end{equation}
Again by Proposition \ref{Betti numbers of symmetric product}, we have:
\begin{equation*}
\begin{split}
&\text{P.P\; of \; } \Sym^3(C)\;\text{is}\; 1+2gx+(1+\big(\begin{smallmatrix}
  2g\\
  2
\end{smallmatrix}\big))x^2+\cdots+x^6,\\
&\text{P.P\; of \; } \Sym^2(C)\;\text{is}\; 1+2gx+(1+\big(\begin{smallmatrix}
  2g\\
  2
\end{smallmatrix}\big))x^2+2gx^3+x^4.
\end{split}
\end{equation*}
Therefore,
\begin{equation}\label{PP of (3,2)}
\text{P.P\; of \; } \Sym^3(C)\times \Sym^2(C)\;\text{is}\; 1+4gx+(2+4g^2+2 \big(\begin{smallmatrix}
  2g\\
  2
\end{smallmatrix}\big))x^2+\cdots+x^{10}.
\end{equation}

Let us take $n=14$ and the partitions $(8,6)$ and $(12,2)$.  Then by Proposition \ref{Betti numbers of symmetric product}, we have:
\begin{equation*}
\begin{split}
&\text{P.P\; of \; } \Sym^8(C)\;\text{is}\; 1+2gx+(1+\big(\begin{smallmatrix}
  2g\\
  2
\end{smallmatrix}\big))x^2+(2g+\big(\begin{smallmatrix}
  2g\\
  3
\end{smallmatrix}\big))x^3+\cdots+x^{16},\\
&\text{P.P\; of \; } \Sym^6(C)\;\text{is}\; 1+2gx+(1+\big(\begin{smallmatrix}
  2g\\
  2
\end{smallmatrix}\big))x^2+(2g+\big(\begin{smallmatrix}
  2g\\
  3
\end{smallmatrix}\big))x^3+\cdots+x^{12}.
\end{split}
\end{equation*}
Therefore,
\begin{equation}\label{PP of (8,6)}
\begin{split}
\text{P.P\; of \; } \Sym^8(C)\times \Sym^6(C)\;\text{is}\; 1+4gx &+(2+4g^2+2 \big(\begin{smallmatrix}
  2g\\
  2
\end{smallmatrix}\big))x^2\\
&+(8g+4g \big(\begin{smallmatrix}
  2g\\
  2
\end{smallmatrix}\big)+2g \big(\begin{smallmatrix}
  2g\\
  3
\end{smallmatrix}\big))x^3+\cdots+x^{28}.
\end{split}
\end{equation}
Again by Proposition \ref{Betti numbers of symmetric product}, we have :
\begin{equation*}
\begin{split}
&\text{P.P\; of \; } \Sym^{12}(C)\;\text{is}\; 1+2gx+(1+\big(\begin{smallmatrix}
  2g\\
  2
\end{smallmatrix}\big))x^2+(2g+\big(\begin{smallmatrix}
  2g\\
  3
\end{smallmatrix}\big))x^3+\cdots+x^{24},\\
&\text{P.P\; of \; } \Sym^2(C)\;\text{is}\; 1+2gx+(1+\big(\begin{smallmatrix}
  2g\\
  2
\end{smallmatrix}\big))x^2+2gx^3+x^{4}.
\end{split}
\end{equation*}
Therefore,
\begin{equation}\label{PP of (12,2)}
\begin{split}
\text{P.P\; of \; } \Sym^{12}(C)\times \Sym^2(C)\;\text{is}\; 1+4gx &+(2+4g^2+2 \big(\begin{smallmatrix}
  2g\\
  2
\end{smallmatrix}\big))x^2\\
&+(8g+4g \big(\begin{smallmatrix}
  2g\\
  2
\end{smallmatrix}\big)+ \big(\begin{smallmatrix}
  2g\\
  3
\end{smallmatrix}\big))x^3+\cdots+x^{28}.
\end{split}
\end{equation}

In both of these examples, we considered two parts.  Now let us work with 3 parts.  Let us take $n=11$ and the partitions $(5,4,2)$ and $(4,4,3)$.  Then by Proposition \ref{Betti numbers of symmetric product}, we have:
\begin{equation}\label{3 parts_first stage}
\begin{split}
&\text{P.P\; of \; } \Sym^5(C)\;\text{is}\; 1+2gx+(1+\big(\begin{smallmatrix}
  2g\\
  2
\end{smallmatrix}\big))x^2+(2g+\big(\begin{smallmatrix}
  2g\\
  3
\end{smallmatrix}\big))x^3+\cdots+x^{10},\\
&\text{P.P\; of \; } \Sym^4(C)\;\text{is}\; 1+2gx+(1+\big(\begin{smallmatrix}
  2g\\
  2
\end{smallmatrix}\big))x^2+(2g+\big(\begin{smallmatrix}
  2g\\
  3
\end{smallmatrix}\big))x^3+\cdots+x^8,\\
&\text{P.P\; of \; } \Sym^3(C)\;\text{is}\; 1+2gx+(1+\big(\begin{smallmatrix}
  2g\\
  2
\end{smallmatrix}\big))x^2+(2g+\big(\begin{smallmatrix}
  2g\\
  3
\end{smallmatrix}\big))x^3+\cdots+x^6,\\
&\text{P.P\; of \; } \Sym^2(C)\;\text{is}\; 1+2gx+(1+\big(\begin{smallmatrix}
  2g\\
  2
\end{smallmatrix}\big))x^2+2gx^3+x^4.
\end{split}
\end{equation}
From \eqref{3 parts_first stage} we obtain :
\begin{equation}\label{3 parts_second stage_first substage}
\begin{split}
\text{P.P\; of \; } \Sym^{4}(C)\times \Sym^2(C)\;\text{is}\; 1+4gx &+(2+4g^2+2 \big(\begin{smallmatrix}
  2g\\
  2
\end{smallmatrix}\big))x^2\\
&+(8g+4g \big(\begin{smallmatrix}
  2g\\
  2
\end{smallmatrix}\big)+ \big(\begin{smallmatrix}
  2g\\
  3
\end{smallmatrix}\big))x^3+\cdots+x^{12}.
\end{split}
\end{equation}
\begin{equation}\label{3 parts_second stage_second substage}
\begin{split}
\text{P.P\; of \; } \Sym^{4}(C)\times \Sym^4(C)\;\text{is}\; 1+4gx &+(2+4g^2+2 \big(\begin{smallmatrix}
  2g\\
  2
\end{smallmatrix}\big))x^2\\
&+(8g+4g \big(\begin{smallmatrix}
  2g\\
  2
\end{smallmatrix}\big)+ 2 \big(\begin{smallmatrix}
  2g\\
  3
\end{smallmatrix}\big))x^3+\cdots+x^{16}.
\end{split}
\end{equation}
From \eqref{3 parts_first stage} and \eqref{3 parts_second stage_first substage}, we obtain :
\begin{equation}\label{PP of (5,4,2)}
\begin{split}
\text{P.P\; of \; } \Sym^{5}(C)\times \Sym^{4}(C)&\times \Sym^2(C)\;\text{is}\; 1+6gx+(3+12g^2+3 \big(\begin{smallmatrix}
  2g\\
  2
\end{smallmatrix}\big))x^2\\
&+(18g+8g^3+12g \big(\begin{smallmatrix}
  2g\\
  2
\end{smallmatrix}\big)+ 2 \big(\begin{smallmatrix}
  2g\\
  3
\end{smallmatrix}\big))x^3+\cdots+x^{22}.
\end{split}
\end{equation}
From \eqref{3 parts_first stage} and \eqref{3 parts_second stage_second substage}, we obtain :
\begin{equation}\label{PP of (4,4,3)}
\begin{split}
\text{P.P\; of \; } \Sym^{4}(C)\times \Sym^{4}(C)&\times \Sym^3(C)\;\text{is}\; 1+6gx+(3+12g^2+3 \big(\begin{smallmatrix}
  2g\\
  2
\end{smallmatrix}\big))x^2\\
&+(18g+8g^3+12g \big(\begin{smallmatrix}
  2g\\
  2
\end{smallmatrix}\big)+ 3 \big(\begin{smallmatrix}
  2g\\
  3
\end{smallmatrix}\big))x^3+\cdots+x^{22}.
\end{split}
\end{equation}
\begin{remark}\label{equality and difference of Betti numbers}
\begin{enumerate}
\item It can be noted that the $2$nd Betti numbers of the multi symmetric product of type $[(4,1),5]$ and $[(3,2),5]$ are different (cf. \eqref{PP of (4,1)} and \eqref{PP of (3,2)}) whenever $g\geq 1$.  Similarly, whenever $g\geq 2$,  the $3$rd Betti numbers of both the pairs of the multi symmetric product of type $[(8,6),14]$ \& $[(12,2),14]$  and $[(5,4,2),11]$ \& $[(4,4,3),11]$ are different (cf. (\eqref{PP of (8,6)} \& \eqref{PP of (12,2)}) and (\eqref{PP of (5,4,2)} \& \eqref{PP of (4,4,3)})).
\item It can be noted that at least the first two Betti numbers, i.e, the $0$th and $1$st Betti numbers of the multi symmetric product of type $[(4,1),5]$ and $[(3,2),5]$ are same (cf. \eqref{PP of (4,1)} and \eqref{PP of (3,2)}), whereas the $0$th, $1$st and $2$nd Betti numbers of both the pairs of the multi symmetric product of type $[(8,6),14]$ \& $[(12,2),14]$ and $[(5,4,2),11]$ \& $[(4,4,3),11]$ are same (cf. (\eqref{PP of (8,6)} \& \eqref{PP of (12,2)}) and (\eqref{PP of (5,4,2)} \& \eqref{PP of (4,4,3)})). 
\end{enumerate}
\end{remark}
We have the following lemma motivated by part $(2)$ of Remark \ref{equality and difference of Betti numbers}.
\begin{lemma}\label{the equal summands}
Let $m$ and $n$ be two integers satisfying $m>n\geq 1$.  Let $B_i^m$(respectively $B_j^n$) be the $i$-th Betti number of $\Sym^m(C)$ (respectively the $j$-th Betti number of $\Sym^n(C)$) for all $1\leq i \leq 2m$ (respectively $1\leq j \leq 2n$).  Then $B_i^m=B_i^n$ for all $1\leq i \leq n$.
\end{lemma}
\begin{proof}
Follows directly from Proposition \ref{Betti numbers of symmetric product}.
\end{proof}
Now we have the following definition and a couple of results followed by that which we need to handle some trivial cases while proving the upcoming proposition about the classification of multi symmetric products corresponding to distinct partitions of same length.
\begin{definition}
Two varieties $X$ and $Y$ are said to be Picard independent if given any $\mathcal{L}\in \Pic(X\times Y)$, there exist $\mathcal{L}_1\in \Pic(X)$ and $\mathcal{L}_2\in \Pic(Y)$ such that $\mathcal{L}=p_1^{\ast}\mathcal{L}_1\otimes p_2^{\ast}\mathcal{L}_2$, where $p_1:X\times Y\rightarrow X$ and $p_2:X\times Y\rightarrow Y$ are usual projections.  
\end{definition}
\begin{lemma}\label{Cancellation property of complete varieties}
Let $M$, $V$ and $W$ be varieties such that $M\times V\cong M\times W$.  If $M$ is projective and $M$ and $V$ are Picard independent, then $V\cong W$.
\end{lemma}
\begin{proof}
See \cite[Theorem 6, p.~120]{Fuj}. 
\end{proof}
\begin{corollary}\label{Cancellation property to ignore common components}
Let $V$ and $W$ be two smooth projective varieties such that $\mathbb{P}^n\times V \cong \mathbb{P}^n\times W$.  Then $V\cong W$
\end{corollary}
\begin{proof}
Follows directly from Lemma \ref{Cancellation property of complete varieties} and \cite[Chapter \RomanNumeralCaps{2}, Corollary 6.16, p.~145 \& Exercise 6.1, p.~146]{Ha}.
\end{proof}
Now we are in a situation to prove that multi symmetric products corresponding to distinct partitions of an integer of same length are not isomorphic.  In fact, generalizing the observations made in Remark \ref{equality and difference of Betti numbers},  we have the following proposition when the smallest part among all the parts of the partitions involved is bounded above.
\begin{proposition}\label{non isomorphic multi symmetric products corresponding to partitions of same length having smaller parts}
Let $C$ be a smooth projective curve over $\mathbb{C}$ of genus $g\geq 1$.  Let $n$ be a positive integer, and $(n_1,n_2,\ldots,n_r)$ and $(m_1,m_2,\ldots,m_r)$ two distinct partitions of $n$ of same length.  Then the multi symmetric product of $C$ of type $[(n_1,n_2,\ldots,n_r),n]$ and $[(m_1,m_2,\ldots,m_r),n]$ are not isomorphic whenever  $\min\{n_r,m_r\}\leq 2g-1$.
\end{proposition}
\begin{proof}
To prove the proposition, it is enough to show that at least one of the Betti numbers of the multi symmetric product of $C$ of type $[(n_1,n_2,\ldots,n_r),n]$ and $[(m_1,m_2,\ldots,m_r),n]$ is different.  Equivalently, it is enough to show that the at least one of the coefficients of the Poincar\'e polynomials of these two spaces are different.\\   
\textit{First Case :} $(n_1,n_2,\ldots,n_r)$ and $(m_1,m_2,\ldots,m_r)$ with no common parts\\
We have that $n_i\neq m_j$ for all $1\leq i,j \leq r$.  Moreover, we can assume that $n_1< m_1$, w.l.o.g.  Hence we have :
\begin{equation*}
\begin{split}
&1\leq n_1\leq n_2 \leq \cdots \leq n_r,\\
&1\leq m_1\leq m_2 \leq \cdots \leq m_r,\\
&n_1<m_1.
\end{split}
\end{equation*}
Clearly, Poincar\'e polynomials of the multi symmetric products are of degree $2n$.  Let us look at the coefficients of $x^{n_1+1}$, $x$ being the indeterminant of the polynomials.  The coefficient of $x^{n_1+1}$ in the Poincar\'e polynomial of the multi symmetric product $\Sym^{n_1}(C)\times \Sym^{n_2}(C)\times \cdots \times \Sym^{n_r}(C)$ arise as the sum of the product of the coefficients of lesser powers of $x$ in the Poincar\'e polynomials of $\Sym^{n_i}(C)$'s such that those powers add up to $n_1+1$.  That is,
\begin{equation}\label{coefficient in first type}
\begin{split}
\text{Coefficient\; of \; }&x^{n_1+1}\text{\; in\; the \;P.P \;of \; }\Sym^{n_1}(C)\times \Sym^{n_2}(C)\times \cdots \times \Sym^{n_r}(C)\\
&=\sum_{t_1+t_2+\cdots+t_r=n_1+1}(\prod_{i=1}^r\text{Coefficient\; of \; }x^{t_i}\text{\; in\; the \;P.P \;of \; }\Sym^{n_i}(C)).
\end{split}
\end{equation} 
Similarly, we have :
\begin{equation}\label{coefficient in second type}
\begin{split}
\text{Coefficient\; of \; }&x^{n_1+1}\text{\; in\; the \;P.P \;of \; }\Sym^{m_1}(C)\times \Sym^{m_2}(C)\times \cdots \times \Sym^{m_r}(C)\\
&=\sum_{t_1+t_2+\cdots+t_r=n_1+1}(\prod_{i=1}^r\text{Coefficient\; of \; }x^{t_i}\text{\; in\; the \;P.P \;of \; }\Sym^{m_i}(C)).
\end{split}
\end{equation}   
Now as the number of parts in both the partitions are same, which is $r$, not only the number of summands on the r.h.s of \eqref{coefficient in first type} and \eqref{coefficient in second type} are same but also the number of terms which are getting multiplied in each such summand are also same.\\ 
By Proposition \ref{Betti numbers of symmetric product}, we have :
\begin{equation}\label{coefficient corresponding to smallest part}
\begin{split}
\text{Coefficient\; of \; } &x^{n_1+1} \text{\; in\; the\; P.P\; of \; } \Sym^{n_1}(C)\\
&= \left\{ \begin{array}{ll}
\big(\begin{smallmatrix}
  2g\\
  n_1-1
\end{smallmatrix}\big)+\big(\begin{smallmatrix}
  2g\\
  n_1-3
\end{smallmatrix}\big)+\cdots+\big(\begin{smallmatrix}
  2g\\
  0
\end{smallmatrix}\big),& \mbox{if $n_1$ is odd};\\ \\ \big(\begin{smallmatrix}
  2g\\
  n_1-1
\end{smallmatrix}\big)+\big(\begin{smallmatrix}
  2g\\
  n_1-3
\end{smallmatrix}\big)+\cdots+\big(\begin{smallmatrix}
  2g\\
  1
\end{smallmatrix}\big), &
\mbox{if $n_1$ is even}.\end{array} \right.
\end{split}
\end{equation}
Similarly, as $n_1<m_1$ or equivalently $n_1+1\leq m_1$, by Proposition \ref{Betti numbers of symmetric product}, we have :
\begin{equation}\label{coefficient corresponding to the other one}
\begin{split}
\text{Coefficient\; of \; } &x^{n_1+1} \text{\; in\; the\; P.P\; of \; } \Sym^{m_1}(C)\\
&= \left\{ \begin{array}{ll}
\big(\begin{smallmatrix}
  2g\\
  n_1+1
\end{smallmatrix}\big)+\big(\begin{smallmatrix}
  2g\\
  n_1-1
\end{smallmatrix}\big)+\cdots+\big(\begin{smallmatrix}
  2g\\
  0
\end{smallmatrix}\big),& \mbox{if $n_1$ is odd};\\ \\ \big(\begin{smallmatrix}
  2g\\
  n_1+1
\end{smallmatrix}\big)+\big(\begin{smallmatrix}
  2g\\
  n_1-1
\end{smallmatrix}\big)+\cdots+\big(\begin{smallmatrix}
  2g\\
  1
\end{smallmatrix}\big), &
\mbox{if $n_1$ is even}.\end{array} \right.
\end{split}
\end{equation}
Therefore, whenever $n_1\leq 2g-1$, by \eqref{coefficient corresponding to smallest part} and \eqref{coefficient corresponding to the other one}, we have :
\begin{equation}\label{at least one coefficient is different}
\begin{split}
\text{Coefficient\; of \; }& x^{n_1+1} \text{\; in \; the\; P.P\; of \; } \Sym^{n_1}(C)\\
&\neq \text{Coefficient\; of \; } x^{n_1+1} \text{\; in\; the\; P.P\; of \; } \Sym^{m_1}(C).
\end{split}
\end{equation}
By \eqref{at least one coefficient is different}, all the summands on the r.h.s of \eqref{coefficient in first type} and \eqref{coefficient in second type} are equal (by Lemma \ref{the equal summands}) except the summand 
\begin{equation*}
\begin{split}
&\text{Coefficient\; of \; } x^{n_1+1} \text{\; in\; the\; P.P\; of \; } \Sym^{n_1}(C)\\
&\times \text{\;constant\; coefficient \;in\; the\; P.P \; of \; } \Sym^{n_2}(C)\times \cdots \\
&\times \text{\;constant\; coefficient\; in\; the\; P.P \; of \; } \Sym^{n_r}(C)\\
&=\text{Coefficient\; of \; } x^{n_1+1} \text{\; in\; the\; P.P\; of \; } \Sym^{n_1}(C)\times \big(\begin{smallmatrix}
  2g\\
  0
\end{smallmatrix}\big)\times \cdots \times \big(\begin{smallmatrix}
  2g\\
  0
\end{smallmatrix}\big)\\
&=\text{Coefficient\; of \; } x^{n_1+1} \text{\; in\; the\; P.P\; of \; } \Sym^{n_1}(C)
\end{split}
\end{equation*}
in the coefficient of $x^{n_1+1}$ in the Poincar\'e polynomial of $\Sym^{n_1}(C)\times \Sym^{n_2}(C)\times \cdots \times \Sym^{n_r}(C)$ and the summand 
\begin{equation*}
\begin{split}
&\text{Coefficient\; of \; } x^{n_1+1} \text{\; in\; the\; P.P\; of \; } \Sym^{m_1}(C)\\
&\times \text{\;constant\; coefficient \;in\; the\; P.P\; of \; } \Sym^{m_2}(C)\times \cdots \\
&\times \text{\;constant\; coefficient\; in\; the\; P.P\; of \; } \Sym^{m_r}(C)\\
&=\text{Coefficient\; of \; } x^{n_1+1} \text{\; in\; the\; P.P\; of \; } \Sym^{m_1}(C)\times \big(\begin{smallmatrix}
  2g\\
  0
\end{smallmatrix}\big)\times \cdots \times \big(\begin{smallmatrix}
  2g\\
  0
\end{smallmatrix}\big)\\
&=\text{Coefficient\; of \; } x^{n_1+1} \text{\; in\; the\; P.P\; of \; } \Sym^{m_1}(C)
\end{split}
\end{equation*}
in the coefficient of $x^{n_1+1}$ in the Poincar\'e polynomial of $\Sym^{m_1}(C)\times \Sym^{m_2}(C)\times \cdots \times \Sym^{m_r}(C)$.  Therefore $(n_1+1)$-th Betti number $B_{n_1+1}$ are different for the multi symmetric product of $C$ of type $[(n_1,n_2,\ldots,n_r),n]$ and $[(m_1,m_2,\ldots,m_r),n]$ and hence the assertion follows.\\
\textit{Second Case :} $(n_1,n_2,\ldots,n_r)$ and $(m_1,m_2,\ldots,m_r)$ with at least one common part\\
If some of the parts of the partitions $(n_1,n_2,\ldots,n_r)$ and $(m_1,m_2,\ldots,m_r)$ of $n$ are equal, then the assertion follows from repeated application of Corollary \ref{Cancellation property to ignore common components} and the previous part.
\end{proof}
Let us now check whether the multi symmetric products corresponding to distinct partitions of an integer of same length are isomorphic or not when the smallest part is bounded below.  

Given $r$ many positive integers $d_1, d_2, \cdots, d_r$, we denote the product map $\alpha_{d_{1},P}\times \cdots \times \alpha_{d_{r},P}$ by $\alpha_{d_1, \cdots, d_r,P}$.  That is, we have :
\begin{equation}\label{product of Abel-Jacobi maps}
\begin{split}
\alpha_{d_1, \cdots, d_r,P}:\Sym^{d_1}(C)\times \cdots \times \Sym^{d_r}(C) & \rightarrow J(C)\times \cdots \times J(C)\\
(D_1,\ldots,D_r)&\mapsto \mathcal{O}_C(D_1+\cdots+D_r-(d_1+\cdots+d_r)P).
\end{split}
\end{equation}
We now check that the fibres of this map are multiprojective spaces.  To be precise, we have the following lemma.
\begin{lemma}\label{using multi projective bundleness of multi symmetric products}
Let $C$ be a smooth projective curve of genus $g$.  Let $d_1,\cdots, d_r$ be $r$ many positive integers satisfying $d_i\geq 2g-1$ for all $1\leq i \leq r$.  Then the fibre of the map $\alpha_{d_1, \cdots, d_r,P}$, as in \eqref{product of Abel-Jacobi maps}, is isomorphic to $\mathbb{P}^{d_1-g}\times \cdots \times \mathbb{P}^{d_r-g}$.
\end{lemma}
\begin{proof}
Follows directly from Lemma \ref{using projective bundleness of symmetric products}.
\end{proof}
The following lemma characterises any morphism from a multiprojective space to an abelian variety.
\begin{lemma}\label{constants map from multiprojective space to product of Jacobians}
Let $n$, $d_1,\cdots, d_r$ be positive integers.  Then  
\begin{enumerate}
\item Any morphism from $\mathbb{P}^n$ to an abelian variety is constant.
\item Any morphism from $\mathbb{P}^{d_1}\times \cdots \times \mathbb{P}^{d_r}$ to an abelian variety is constant.
\end{enumerate}  
\end{lemma}
\begin{proof}
\begin{enumerate}
\item Let $\mathcal{A}$ be an abelian variety of dimension $m$ with local coordinates $v_i$, $1\leq i\leq m$ and $f:\mathbb{P}^n\rightarrow \mathcal{A}$ be any morphism.  By $\Omega_{\mathcal{A}}$ and $\Omega_{\mathbb{P}^n}$ we denote the cotangent bundle of the abelian variety $\mathcal{A}$ and $\mathbb{P}^n$ respectively.  As, the tangent bundle of $\mathcal{A}$ is the trivial bundle of rank $m$, so is its cotangent bundle.  Let us consider the basis $\{dv_i\mid 1\leq i\leq m\}$ of $H^0(\mathcal{A}, \Omega_{\mathcal{A}})$.  We now have the following map induced by $f$ :
\begin{equation}\label{pullback of differential forms from Jacobian to Projective line}
\begin{split}
H^0(\mathcal{A},\Omega_{\mathcal{A}})&\rightarrow H^0(\mathbb{P}^n,\Omega_{\mathbb{P}^n})\\
dv_i&\mapsto f^{\ast}dv_i\;\text{for\;all\;}i=1,\cdots, m.
\end{split}
\end{equation} 
Now as $H^0(\mathbb{P}^n,\Omega_{\mathbb{P}^n})=\{0\}$ (cf. \cite[Chapter \RomanNumeralCaps{2}, Example 8.20, p.~182]{Ha}), from \eqref{pullback of differential forms from Jacobian to Projective line} we have $f^{\ast} dv_i=0$ for all $i$.  Therefore, $f$ must be constant.     
\item Follows from the first part.  Alternatively, it can be noted that the multiprojective space $\mathbb{P}^{d_1}\times \cdots \times \mathbb{P}^{d_r}$ is a unirational variety.  Therefore the assertion follows from the fact that any rational map from a unirational variety to an abelian variety is constant, (cf. \cite[Proposition 3.10, p. 20]{Mi}).
\end{enumerate}
\end{proof}
Now we are ready to check whether the multi symmetric products corresponding to distinct partitions of an integer of same length are isomorphic or not when the smallest part is bounded below.  We have the following proposition in that regard. 
\begin{proposition}\label{non isomorphic multi symmetric products corresponding to partitions of same length having bigger parts}
Let $C$ be a smooth projective curve over $\mathbb{C}$ of genus $g$ with $g\geq 1$.  Let $n$ be a positive integer, and $(n_1,n_2,\ldots,n_r)$ and $(m_1,m_2,\ldots,m_r)$ two distinct partitions of $n$ of same length.  Then the multi symmetric product of $C$ of type $[(n_1,n_2,\ldots,n_r),n]$ and $[(m_1,m_2,\ldots,m_r),n]$ are not isomorphic whenever  $\min\{n_r,m_r\}\geq 2g-1$.
\end{proposition}
\begin{proof}
If possible, let there exists an isomorphism $\psi$ from the multi symmetric product of $C$ of type $[(n_1,n_2,\ldots,n_r),n]$ to the multi symmetric product of $C$ of type $[(m_1,m_2,\ldots,m_r),n]$.  Now, by Lemma \ref{using multi projective bundleness of multi symmetric products}, we have the following diagram:
\begin{equation*}\label{multiprojective bundle over the same base}
\begin{tikzcd}
\Sym^{n_1}(C)\times \cdots \times \Sym^{n_r}(C) \arrow[rr,"\psi","\simeq"'] \arrow[ddr, "\alpha_{n_1, \cdots, n_r,P}"']& &
\Sym^{m_1}(C)\times \cdots \times \Sym^{m_r}(C) \arrow[ddl,"\alpha_{m_1, \cdots, m_r,P}"]\\\\
&J(C)\times \cdots \times J(C)&
\end{tikzcd}
\end{equation*}
For any $(\mathcal{L}_1,\ldots,\mathcal{L}_r)\in J(C)\times \cdots \times J(C)$, consider the morphism 
\begin{equation*}\label{the composition map which is actually a constant}
\alpha_{m_1, \cdots, m_r,P}\circ \psi|_{\alpha_{n_1, \cdots, n_r,P}^{-1}(\mathcal{L}_1,\ldots,\mathcal{L}_r)}
\end{equation*}
given as follows
\begin{equation}\label{description of the composition map through diagram}
\begin{tikzcd}
\alpha_{n_1, \cdots, n_r,P}^{-1}(\mathcal{L}_1,\ldots,\mathcal{L}_r)\arrow[d, mapsto]\arrow[rrr,"\psi|_{\alpha_{n_1, \cdots, n_r,P}^{-1}(\mathcal{L}_1,\ldots,\mathcal{L}_r)}"]&&&\Sym^{m_1}(C)\times \cdots \times \Sym^{m_r}(C)\arrow[d,"\alpha_{m_1, \cdots, m_r,P}"]\\
\{(\mathcal{L}_1,\ldots,\mathcal{L}_r)\}\arrow[rrr, hook]&& &J(C)\times \cdots \times J(C)
\end{tikzcd}
\end{equation}
Now, by Lemma \ref{using multi projective bundleness of multi symmetric products}, we have that the morphism $\alpha_{m_1, \cdots, m_r,P}\circ \psi|_{\alpha_{n_1, \cdots, n_r,P}^{-1}(\mathcal{L}_1,\ldots,\mathcal{L}_r)}$, as in \eqref{description of the composition map through diagram},  is a morphism from $\mathbb{P}^{n_1-g}\times \cdots \times \mathbb{P}^{n_r-g}$ to the abelian variety $J(C)^r$ and therefore is constant, say $(\mathcal{M}_1,\ldots,\mathcal{M}_r)$, by Lemma \ref{constants map from multiprojective space to product of Jacobians}.  So it must factor through $\alpha_{m_1, \cdots, m_r,P}^{-1}(\mathcal{M}_1,\ldots,\mathcal{M}_r)$ (cf. \eqref{must factor through}).   
\begin{equation}\label{must factor through}
\begin{tikzcd}
\alpha_{n_1, \cdots, n_r,P}^{-1}(\mathcal{L}_1,\ldots,\mathcal{L}_r)\arrow[rrr,hook,"\psi|_{\alpha_{n_1, \cdots, n_r,P}^{-1}(\mathcal{L}_1,\ldots,\mathcal{L}_r)}"]&&&\alpha_{m_1, \cdots, m_r,P}^{-1}(\mathcal{M}_1,\ldots,\mathcal{M}_r) \arrow[d, mapsto]\\
&& &\{(\mathcal{M}_1,\ldots,\mathcal{M}_r)\}\subseteq J(C)^r
\end{tikzcd}
\end{equation}
So, by Lemma \ref{using multi projective bundleness of multi symmetric products}, the inclusion as in \eqref{must factor through}, is actually an inclusion between two multiprojective spaces (cf. \eqref{inclusions between product of projective spaces of same dimension}).
\begin{equation}\label{inclusions between product of projective spaces of same dimension}
\begin{tikzcd}
\alpha_{n_1, \cdots, n_r,P}^{-1}(\mathcal{L}_1,\ldots,\mathcal{L}_r)\arrow[d, "\simeq"]\arrow[rrr,hook]&&&\alpha_{m_1, \cdots, m_r,P}^{-1}(\mathcal{M}_1,\ldots,\mathcal{M}_r)\arrow[d, "\simeq"']\\
\mathbb{P}^{n_1-g}\times \cdots \times \mathbb{P}^{n_r-g}\arrow[rrr,hook]&&& \mathbb{P}^{m_1-g}\times \cdots \times \mathbb{P}^{m_r-g}
\end{tikzcd}
\end{equation}
As both $\mathbb{P}^{n_1-g}\times \cdots \times \mathbb{P}^{n_r-g}$ and $\mathbb{P}^{m_1-g}\times \cdots \times \mathbb{P}^{m_r-g}$ have the same dimension $n$, we have:
\begin{equation*}
\mathbb{P}^{n_1-g}\times \cdots \times \mathbb{P}^{n_r-g}\simeq \mathbb{P}^{m_1-g}\times \cdots \times \mathbb{P}^{m_r-g}.
\end{equation*}
Then by Proposition \ref{non isomorphic multiprojective spaces}, we obtain :
\begin{equation*}
m_i=n_i \text{\;for\;all\;} 1\leq i \leq r.
\end{equation*}
But this contradicts the fact that $(n_1,n_2,\ldots,n_r)$ and $(m_1,m_2,\ldots,m_r)$ are two distinct partitions of $n$.  Therefore, any such isomorphism $\psi$ can't exist.  Hence, the assertion follows.  
\end{proof}
\begin{remark}
In Proposition \ref{non isomorphic multi symmetric products corresponding to partitions of same length having smaller parts} and Proposition \ref{non isomorphic multi symmetric products corresponding to partitions of same length having bigger parts}, the condition $g\geq 1$ is necessary as this makes sure that all the parts of the partitions are positive.   
\end{remark}
Altogether, we obtain the following :
\begin{proposition}\label{non isomorphic multi symmetric products corresponding to partitions of same length}
Let $C$ be a smooth projective curve over $\mathbb{C}$ of genus $g$ with $g\geq 1$.  Let $n$ be a positive integer, and $(n_1,n_2,\ldots,n_r)$ and $(m_1,m_2,\ldots,m_r)$ two distinct partitions of $n$ of same length.  Then the multi symmetric product of $C$ of type $[(n_1,n_2,\ldots,n_r),n]$ and $[(m_1,m_2,\ldots,m_r),n]$ are not isomorphic.
\end{proposition}
\begin{proof}
Follows from Proposition \ref{non isomorphic multi symmetric products corresponding to partitions of same length having smaller parts} and Proposition \ref{non isomorphic multi symmetric products corresponding to partitions of same length having bigger parts}.
\end{proof}

Finally, we observe that Theorem \ref{Main Theorem 3} can be modified further.  That is to say, we conclude that the upper bound as in Theorem \ref{Main Theorem 3}, is achieved  by any curve $C$ of any genus $g$ for any positive integer $n$.  To be precise, we obtain the following :  
\begin{theorem}\label{Main Theorem 4}
Let $C$ be a smooth projective curve over $\mathbb{C}$ and $n$ a positive integer.  Let $p(n)$ denote the number of partitions of $n$.  Then there are exactly $p(n)$ many Hilbert schemes $\Hilb^{\underline{n}}_C$ (up to isomorphism) associated to the constant polynomial $n$ and its good partitions $\underline{n}$ satisfying diagonal property.
\end{theorem}
\begin{proof}
Follows from Theorem \ref{Main Theorem 3} and Proposition \ref{non isomorphic multi symmetric products corresponding to partitions of same length}.
\end{proof}
\begin{remark}
It can be noted that the condition on the parts of a partition in the hypothesis of Proposition \ref{non isomorphic multi symmetric products corresponding to partitions of same length having bigger parts} holds by default (cf. Definition \ref{partition of an integer}) for $g=1$.  Therefore, one can conclude that the upper bound, as in Theorem \ref{Main Theorem 3}, is attained by complex elliptic curves as well, only using Proposition \ref{non isomorphic multi symmetric products for different lengths} and Proposition \ref{non isomorphic multi symmetric products corresponding to partitions of same length having bigger parts}.  This can be thought of as an intermediate stage of the modification from Theorem \ref{Main Theorem 3} to Theorem \ref{Main Theorem 4}.
\end{remark}
\section*{Conflict of Interest Statement}
There is no conflict of interest.
\section*{Acknowledgements}
The authors would like to thank Dr. C. Gangopadhyay for many useful discussions.  The authors wish to thank Prof. J. Heinloth for several discussions and suggestions during the conference VBAC 2023 at University of Duisburg-Essen.  The first named author also acknowledge Indian Institute of Science Education and Research Tirupati for financial support (Award No. - IISER-T/Offer/PDRF/A.M./M/01/2021).   

\end{document}